\documentclass[12pt]{amsart}
\RequirePackage[sc]{mathpazo}
\counterwithin{equation}{section}

\usepackage[centering,scale=0.75]{geometry}

\usepackage[all]{xy}
\usepackage{amssymb}
\usepackage{mathrsfs}
\usepackage{graphicx}
\usepackage{xcolor}
\usepackage{tikz}
\usepackage[colorlinks,plainpages,backref,
    linkcolor=blue!80!green,
    citecolor=green!60!blue,
    urlcolor=red!50!black]{hyperref}

\newtheorem{theorem}{Theorem}[section]
\newtheorem{definition}[theorem]{Definition}
\newtheorem{lemma}[theorem]{Lemma}

\newtheorem{prop}[theorem]{Proposition}

\newtheorem{example}[theorem]{Example}
\newtheorem{remark}[theorem]{Remark}

\def\bfN{\mathbf{N}}
\def\bfM{\mathbf{M}}
\def\calO{\mathcal{O}}
\def\calK{\mathcal{K}}
\def\Gr{\mathsf{Gr}}
\def\vv{\mathbf{v}}
\def\ww{\mathbf{w}}
\def\bbV{\mathbb{V}}
\def\bbW{\mathbb{W}}

\def\calR{{\mathcal{R}}}
\def\pt{\mathsf{pt}}
\def\Eu{\operatorname{Eu}}
\def\SO{\operatorname{SO}}
\def\Sp{\operatorname{Sp}}
\def\GL{\operatorname{GL}}

\def\Lie{\operatorname{Lie}}
\def\Diff{\operatorname{\mathsf{Diff}}}
\def\diag{\operatorname{diag}}
\def\Res{\operatorname{Res}}
\def\Bil{\operatorname{Bil}}
\def\Hom{\operatorname{Hom}}

\usepackage{etoolbox}

\AtBeginDocument{%
  \let\originalbibitem\bibitem
  \renewcommand{\bibitem}[2][]{%
    \ifstrequal{#2}{SSoX26}
      {\originalbibitem[SSoX26]{#2}}
      {%
        \ifstrequal{#2}{SSX26}
          {\originalbibitem[SSX26]{#2}}
          {%
            \ifstrempty{#1}
              {\originalbibitem{#2}}
              {\originalbibitem[#1]{#2}}%
          }%
      }%
  }%
}

\allowdisplaybreaks

\begin{document}

\title[Coulomb Branches of Separated Cotangent Type]{Quantized Coulomb Branches of Separated Cotangent Type and Orthosymplectic Quivers}

\author{Yaolong Shen}
\address[Yaolong Shen]{School of Mathematical Sciences, Key Laboratory of MEA (Ministry of Education)
\& Shanghai Key Laboratory of PMMP, East China Normal University, Shanghai 200241,
China}
\email{yaolongshen3555@gmail.com}

\author{Changjian Su}
\address[Changjian Su]{Yau Mathematical Sciences Center, Tsinghua University, Beijing, China}
\email{changjiansu@mail.tsinghua.edu.cn}

\author{Rui Xiong}
\address[Rui Xiong]{Department of Mathematics and Statistics, University of Ottawa, 150 Louis-Pasteur, Ottawa, ON, K1N 6N5, Canada}
\email{xiongrui@umich.edu}

\begin{abstract}
We propose a definition of the quantized Coulomb branches of separated cotangent type, and prove that the corresponding classical construction recovers the non-cotangent Coulomb branch. 
We also obtain a formula for quasi-minuscule monopole operators in arbitrary cotangent type. Applying these results, we compute the monopole operators for orthosymplectic quivers and construct a homomorphism from the shifted twisted Yangian of split ADE type to the corresponding quantized Coulomb branch algebra.
\end{abstract}

\maketitle

\setcounter{tocdepth}{1}
\tableofcontents

\section{Introduction}

For a complex reductive group $G$ and a finite-dimensional representation $\bfN$, Braverman, Finkelberg, and Nakajima constructed the Coulomb branch algebra $\mathcal{A}_{\hbar}(G,\bfN)$ by equivariant Borel--Moore homology of the BFN space, with multiplication given by convolution \cite{Nak16,BFN18}.  This construction gives a quantization of the Coulomb branch associated with the symplectic matter representation $\bfN\oplus\bfN^*$.  Its realization inside an algebra of difference operators is particularly useful: geometrically defined monopole classes become explicit difference operators, and in many examples these operators realize generators of shifted Yangians and related quantum groups; see, for example, \cite{BFN19,FT19,KWWY14,SSX26,SSoX26,wang2025quivers,We19}.

The cotangent assumption is nevertheless restrictive.  A general symplectic representation $\bfM$ of $G$ need not admit a $G$-invariant Lagrangian, and therefore need not be of the form $\bfN\oplus\bfN^*$.  Braverman, Dhillon, Finkelberg, Raskin, and Travkin extended the classical construction to such non-cotangent representations, subject to an anomaly cancellation condition, by means of a ring object in a twisted derived Satake category \cite{BDFRT} (see \cite{dedushenko2025coulomb,teleman2022coulomb} for other approaches).  Their construction produces a commutative Coulomb branch algebra, but a quantized algebra that is both explicit and suitable for computations is not presently available in comparable generality.

Orthosymplectic quivers provide a basic and important source of this difficulty. It first appeared in \cite{KP82} to realize the nilpotent orbits in the classical Lie algebras. The gauge groups for orthosymplectic quivers are products of special orthogonal and symplectic groups, while their edge and framing spaces form a symplectic matter representation that is naturally of non-cotangent type. Such quivers arise naturally in the study of supersymmetric boundary conditions and three-dimensional mirror symmetry; see for example \cite{GW09}. The corresponding non-quantized Coulomb branches are being studied by Finkelberg, Hanany, and Nakajima \cite{FHN}, while the Higgs branch are studied in \cite{LYQ,nakajima2025instantons}. On the other hand, one expects the corresponding quantized Coulomb branches to be closely related to shifted twisted Yangians, see \cite{LWW}.  The purpose of this paper is to develop a geometric and computational framework for this problem.  Our construction has three main ingredients, which we now describe.

\subsection*{Quantized Coulomb branches of separated cotangent type}

We first isolate a class of non-cotangent representations for which a quantization can be constructed from ordinary BFN algebras.  Let
\[
G=G_1\times G_2
\]
and let $\bfM$ be a symplectic representation of $G$.  We say that $(G,\bfM)$ is of \emph{separated cotangent type} if the restriction of $\bfM$ to each factor $G_i$ is of cotangent type.  If $T_i\subset G_i$ is a maximal torus, set
\[
T=T_1\times T_2,\qquad TG_1=G_1\times T_2,
\qquad TG_2=T_1\times G_2.
\]
After choosing suitable Lagrangians $\bfN_1$ and $\bfN_2$, the pairs $(TG_i,\bfM)$ are of cotangent type.  Their quantized BFN algebras can therefore be embedded, after a natural twist for the second presentation, into the same difference algebra $\Diff_{\hbar}(T)$.

In Definition~\ref{def:noncot}, we define the quantized Coulomb branch of $(G,\bfM)$ by
\[
\mathscr{A}_{\hbar}(G,\bfM)
:=
\mathcal{A}_{\hbar}(TG_1,\bfN_1)_{\operatorname{Frac}H^*_{T_2}(\pt)}
\bigcap
\mathcal{A}_{\hbar}(TG_2,\bfN_2)_{\operatorname{Frac}H^*_{T_1}(\pt)}
\subset \Diff_{\hbar}(T).
\]
The two factors of the intersection impose complementary regularity conditions in the $G_1$- and $G_2$-directions.  The classical justification for this definition is given in Theorem~\ref{thm:compwithBDFRT}.  Under the technical assumption that one of $G_1$ and $G_2$ is semisimple, we prove that the analogous intersection at $\hbar=0$ is isomorphic to the non-cotangent Coulomb branch algebra of \cite{BDFRT}.  This identifies the classical intersection underlying our definition with the established non-cotangent theory, while the difference operator description remains available for quantization and explicit calculation.

Our first main result also includes a practical criterion for producing elements of this algebra.  The two cotangent presentations induce commuting Weyl group actions on $\Diff_{\hbar}(T)$, one of which is twisted by the change of Lagrangian.  Theorem~\ref{thm:noncotsub} shows that a Weyl-invariant element belonging to either one of the unlocalized cotangent Coulomb branch algebras already belongs to $\mathscr{A}_{\hbar}(G,\bfM)$.  More precisely,
\[
\mathcal{A}_{\hbar}(TG_i,\bfN_i)\cap \Diff_{\hbar}(T)^W
\subset \mathscr{A}_{\hbar}(G,\bfM),
\qquad i=1,2.
\]
This criterion is the main tool that allows us to pass from monopole operators in the two auxiliary cotangent theories to operators in the non-cotangent quantized Coulomb branch.

\subsection*{Quasi-minuscule monopole operators}

The second ingredient is a formula for monopole operators in a quantized Coulomb branch of cotangent type.  The localization formula for a dressed monopole operator associated with a minuscule coweight is well known, see \cite{BFN18}.  For a quasi-minuscule coweight $\lambda=\theta^\vee$, however, the affine Grassmannian Schubert variety has an additional zero-dimensional stratum:
\[
\overline{\Gr_{\lambda}}=\Gr_{\lambda}\sqcup\{e\}.
\]
The contribution of this stratum prevents a direct application of the minuscule calculation.

In Section~\ref{sec:quasi} we use the explicit resolution
\[
\mathbb{P}(\mathcal{O}\oplus\mathcal{L})
\longrightarrow \overline{\Gr_{\lambda}}
\]
and construct the corresponding resolved BFN space.  Equivariant localization on this resolution gives Theorem~\ref{thm:quasi}.  In particular, for $f\in\mathbb{C}[\mathfrak{t}]^{W_\lambda}$, it proves that
\[
\sum_{w\in W^\lambda}
w\!\left(
f\,
\frac{\Eu\bigl(z^\lambda\bfN_{\calO}/
(z^\lambda\bfN_{\calO}\cap\bfN_{\calO})\bigr)}
{\Eu(T_\lambda\Gr_\lambda)}
\right)(d_{w\lambda}-1)
\in \mathcal{A}_{\hbar}(G,\bfN).
\]
More precisely, this operator is the image of a geometrically defined class on the resolved BFN space, up to an element of $H^*_{G\times\mathbb{C}^\times_\hbar}(\pt)$.  This result extends the usual minuscule formula and supplies exactly the operator needed at a symplectic vertex of an orthosymplectic quiver.

\subsection*{Orthosymplectic quivers and shifted twisted Yangians}

We then apply the preceding constructions to an orthosymplectic quiver $Q=(Q_0,Q_1,\epsilon)$.  The gauge group and matter representation are
\[
G=\prod_{i\in Q_0}K_{\epsilon_i}(\bbV_i),
\qquad
\bfM=
\bigoplus_{i-j}\bbV_i\otimes\bbV_j
\oplus
\bigoplus_{i\in Q_0}\bbV_i\otimes\bbW_i,
\]
where $K_+=\SO$ and $K_-=\Sp$, and $\epsilon_i=\pm1$ such that adjacent vertices have opposite signs.  We first determine the anomaly cancellation condition (see also \cite{Nak23}): at every symplectic vertex $i$ one must have
\[
\ww_i+\sum_{j-i}\vv_j\equiv0\pmod 2.
\]
Under this condition, $(G,\bfM)$ is of separated cotangent type.  The two auxiliary cotangent presentations are obtained by choosing maximal isotropic subspaces in the orthogonal and symplectic vector spaces.

The first fundamental coweight at an orthogonal vertex is minuscule, whereas the first fundamental coweight at a symplectic vertex is quasi-minuscule.  We can therefore combine the usual minuscule formula, the quasi-minuscule formula of Theorem~\ref{thm:quasi}, and the Weyl-invariance criterion of Theorem~\ref{thm:noncotsub}.  This gives explicit generating series of monopole operators
\[
B_i(u)\in\mathscr{A}^F_{\hbar}(G,\bfM)[[u^{-1}]]
\]
for every vertex $i$, together with explicit Cartan series $H_i(u)$.  The construction is geometric: at an orthogonal vertex the coefficients of $B_i(u)$ are represented by dressed minuscule monopole classes, while at a symplectic vertex they are represented, up to terms from the Gelfand--Tsetlin subalgebra, by classes obtained from the quasi-minuscule resolution described above.

Suppose now that the underlying graph of $Q$ is a Dynkin diagram of ADE type, and let $\mathfrak{g}$ be the corresponding simple Lie algebra.  Define the even coweight
\[
\mu=\sum_{i\in Q_0}
\bigl(\ww_i\Lambda_i-\vv_i\alpha_i^\vee
+2\epsilon_i\Lambda_i\bigr).
\]
Our final main result, Theorem~\ref{thm:main}, states that the assignment
\[
b_i(u)\longmapsto B_i(u),
\qquad
h_i(u)\longmapsto H_i(u)
\]
defines a homomorphism from the shifted twisted Yangian of split type to the localized quantized Coulomb branch algebra:
\[
\Psi:{}^{\iota}\mathbf{Y}_{\mu}(\mathfrak{g})\otimes H_F^*(\pt)
\longrightarrow
\mathscr{A}^F_{\hbar}(G,\bfM)
\otimes_{\mathbb{C}[\hbar]}\mathbb{C}(\hbar).
\]
This gives a geometric realization of the Yangian generators and provides evidence for the proposed relationship between orthosymplectic Coulomb branches and shifted twisted Yangians in \cite{LWW}. 

Related iGKLO-type difference-operator representations have been constructed in \cite{lu2025shifted,bartlett2025gklo}; the former treats general quasi-split ADE types, while the latter treats type $\mathsf{AI}$ (without flavor symmetry).  
Our formulas are closer to the former one; see Remark \ref{rem:gklo} for a comparison. 
The distinguishing feature of the present construction is its geometric origin: our operators are obtained from explicit minuscule and quasi-minuscule monopole classes by equivariant localization, and their role in the non-cotangent quantized Coulomb branch follows from the separated-cotangent criterion. 
In this sense, our formulas are intrinsic from the viewpoint of Coulomb branch geometry.

\subsection*{Organization of the paper}

In Section~\ref{sec:cotCB} we review quantized Coulomb branches of cotangent type and derive the quasi-minuscule monopole formula.  In Section~\ref{sec:noncotCB} we introduce Coulomb branches of separated cotangent type, define their quantization, prove the Weyl-invariance criterion, and compare the classical construction with that of \cite{BDFRT}.  In Section~\ref{sec:orthosymp} we analyze orthosymplectic quivers, verify the anomaly cancellation condition, and compute the monopole series $B_i(u)$ and the Cartan series $H_i(u)$.  In Section~\ref{sec:yangian} we recall the Drinfeld presentation of the shifted twisted Yangian of split ADE type and state the homomorphism to the quantized Coulomb branch.  The defining relations are verified in Section~\ref{sec:veri}.

\subsection*{Acknowledgment}
We would like to thank Michael Finkelberg, Kang Lu, Hiraku Nakajima,  Weiqiang Wang, Alex Weeks, and Ruotao Yang 
for useful discussions.
CS is supported by the National Key R\&D Program of China (No. 2025YFA1017400). YS is partially supported by the Fields Institute and the Science and Technology Commission of Shanghai Municipality (No. 22DZ2229014).

\section{Coulomb branches of cotangent type}\label{sec:cotCB}

In this section, we review the Coulomb branch of cotangent type \cite{BFN18}, including its definition and methods of computation. 
The new result of this section is a general formula for the monopole operator associated with a quasi-minuscule coweight.

\subsection{Definition}
Let $G$ be a complex connected reductive group with a Borel subgroup $B$ and a maximal torus $T$. Let $R$ be the set of roots with positive roots denoted by $R^+$. Let $X_*(T)$ be the cocharacter lattice of $T$ with dominant ones denoted by $X_*(T)^+$. Recall the dominance order on $X_*(T)^+$: for any $\lambda,\mu\in X_*(T)^+$, $\lambda\leq \mu$ iff $\mu-\lambda$ is a nonnegative linear combination of the positive coroots.
Let
$\calO=\mathbb{C}[[z]]$, 
$\calK=\mathbb{C}(\!(z)\!)$. We write $G_\calK:=G(\calK)$, $G_\calO:=G(\calO)$ and $T_\calK:=T(\calK)$, respectively. Any cocharacter $\lambda:\mathbb{G}_m\rightarrow T$ gives a homomorphism $\mathcal{K}^*\rightarrow T_\mathcal{K}$, and we let $z^\lambda$ denote the image of $z\in \mathcal{K}^*$.

Recall the affine Grassmannian is defined to be $\Gr_G=G_\calK/G_\calO$. By the Cartan decomposition,
\[\Gr_G=\bigsqcup_{\lambda\in X_*(T)^+}\Gr_\lambda,\]
where $\Gr_\lambda:=G_\calO z^\lambda G_\calO/G_\calO$. Its closure $\overline{\Gr_\lambda}=\bigsqcup\limits_{\mu\in X_*(T)^+,\mu\leq \lambda}\Gr_\mu$. Moreover, each $\Gr_\lambda$ is an affine bundle over the partial flag variety $G/P_\lambda$, where $P_\lambda$ is the parabolic subgroup generated by $T$ and the root subgroups $U_\alpha$ for any $\alpha\in R$ satisfying $\langle\lambda, \alpha\rangle\leq 0$.

Let $\bfN$ be a complex representation of $G$. Let $\bfN_\calK:=\bfN(\calK)$ and $\bfN_\calO:=\bfN(\calO)$, respectively.
We consider
$$\mathcal{T}_{G,\bfN}
=G_\calK\times_{G_\calO} \bfN_\calO
=\{(gG_\calO,x):x\in g\bfN_\calO\}\subset 
\Gr_G\times \bfN_\calK.$$
The (classical) \emph{BFN space} $\calR_{G,\bfN}$ \cite{BFN18} is defined by the pullback square
$$\xymatrix{
\calR_{G,\bfN} \ar@{^{(}->}[r] \ar@{^{(}->}[d] & \Gr_G\times \bfN_\calO \ar@{^{(}->}[d]\\
\mathcal{T}_{G,\bfN} \ar@{^{(}->}[r]&\Gr_G\times \bfN_\calK.}$$
Written explicitly, we have
$$\calR_{G,\bfN} = \{(gG_\calO,x): x\in g\bfN_\calO\cap \bfN_\calO\}\subset \Gr_G\times \bfN_\calO.$$
Following \cite[Section 2(i)]{BFN18}, we consider the $\mathbb{C}^\times$ action on $\calR_{G,\bfN}$, which rotates $z\in \calO$ by weight $1$ and scales $\bfN$ by weight $\frac{1}{2}$ simultaneously \footnote{We can take a double cover of the $\mathbb{C}^\times$ so that this action is well defined.}. 
Let $\hbar$ be the equivariant parameter of this $\mathbb{C}^*$, and denote this $\mathbb{C}^*$ by $\mathbb{C}^\times_\hbar$. Hence,  
$H_{\mathbb{C}^\times_\hbar}^*(\pt)=\mathbb{C}[\hbar]$. 
The \emph{(quantized) Coulomb branch algebra} is defined to be the Borel--Moore homology
$$\mathcal{A}_\hbar(G,\bfN) = 
H^{G_{\calO}\rtimes\mathbb{C}_\hbar^\times}_*(\calR_{G,\bfN}).$$
endowed with the convolution product.
The subalgebra $H^{G_\calO\rtimes \mathbb{C}_\hbar^\times}(\pt)\subset \mathcal{A}_\hbar(G,\bfN)$ is called the \emph{Gelfand--Tsetlin subalgebra}. 

Let $W$ be the Weyl group, and let $\mathfrak{t}$ be the Lie algebra of $T$, and $\mathbb{A}^1$ be the Lie algebra of $\mathbb{C}^*_\hbar$.
The Coulomb branch algebra can be embedded into the difference algebra of $T$ as follows. 
We define the \emph{difference algebra} $\Diff_\hbar(T)$ of $T$ to be 
$$\Diff_\hbar(T)=\mathbb{C}(\mathfrak{t}\times \mathbb{A}^1)\rtimes X_*(T),$$
where $\mathbb{C}(\mathfrak{t}\times \mathbb{A}^1)$ is the field of rational functions over $\mathfrak{t}\times \mathbb{A}^1$. 
More precisely, $\Diff_\hbar(T)$ is the free $\mathbb{C}(\mathfrak{t}\times \mathbb{A}^1)$-module with basis $d_\mu$ for cocharacters $\mu\in X_*(T)$ such that 
$$(f(t,\hbar)d_\lambda)\cdot (g(t,\hbar)d_\mu)=f(t,\hbar)g(t+\lambda\hbar,\hbar)d_{\lambda+\mu}.$$
We could view elements in $\Diff_\hbar(T)$ as difference operators on $\mathbb{C}(\mathfrak{t}\times \mathbb{A}^1)$.  
By \cite[Lemma 5.10]{BFN18}, we have an embedding
$
\iota_*:\mathcal{A}_\hbar(T,\bfN)^W
\hookrightarrow
\mathcal{A}_\hbar(G,\bfN)$.
By \cite[Lemma 5.9]{BFN18}, it is an isomorphism after base changing to the fraction field of $H_T^*(\pt)$. 
By \cite[Lemma 5.11]{BFN18}, we have the following embedding
$
z^*:\mathcal{A}_\hbar(T,\bfN)
\hookrightarrow
\mathcal{A}_\hbar(T,0)$. 
The explicit computation in \cite[Section 4]{BFN18} implies $\mathcal{A}_\hbar(T,0)\hookrightarrow \Diff_\hbar(T)$ with 
$H_{T_\calO\times \mathbb{C}_\hbar^\times}^*(\pt)$ identified with the $\mathbb{C}[\mathfrak{t}\times \mathbb{A}^1]$. 
As a result, the composition of the above maps gives 
an embedding of algebras
\begin{equation}\label{eq:diff}
z^*\circ (\iota_*)^{-1}:\mathcal{A}_\hbar(G,\bfN)\stackrel{\subset}\longrightarrow \Diff_\hbar(T). 
\end{equation}
In this paper, we will identify 
$\mathcal{A}_\hbar(G,\bfN)$ as a subalgebra of the difference algebra $\Diff_\hbar(T)$. 

\subsection{Minuscule monopole operators}
By the definition of $\calR_{G,\bfN}$, there is a natural projection 
$$\pi:\calR_{G,\bfN}\longrightarrow \Gr_G.$$
Let us denote 
\begin{equation}\label{equ:Rlambda}
    \calR_{G,\bfN,\lambda}=\pi^{-1}(\Gr_\lambda).
\end{equation}
Note that we can identify
$
H_*^{G_\calO}(\calR_{G,\bfN,\lambda}) \simeq H_*^{G_\calO}(\Gr_\lambda) \simeq \mathbb{C}[\mathfrak{t}]^{W_\lambda}$ where $W_\lambda$ is the stabilizer group of $\lambda$ in the Weyl group $W$. 

Recall a cocharacter $\lambda$ is called \emph{minuscule} if 
$\langle \lambda,\alpha\rangle\in \{0,1\}$
for all positive roots $\alpha$. Equivalently, it is a nonzero minimal dominant cocharacter under the dominant order. Hence, if $\lambda$ is minuscule, $\Gr_\lambda$ is closed, and it is isomorphic to $G/P_\lambda$.
For $f\in \mathbb{C}[\mathfrak{t}]^{W_\lambda}\simeq H_*^{G_\calO}(\calR_{G,\bfN,\lambda})$, we consider the following element 
$$f[\calR_{G,\bfN,\lambda}]\in \mathcal{A}_\hbar(G,\bfN),$$
which is called the \textit{dressed minuscule monopole operator}.
Under the embedding in \eqref{eq:diff}, we have the following explicit formula for it:
\begin{equation}\label{eq:minu}
z^*\circ (\iota_*)^{-1}\big(f[\calR_{G,\bfN,\lambda}]\big)=\sum_{w\in W^\lambda}
w\left(f\cdot \frac{\Eu\left(
z^\lambda\bfN_\calO/(z^\lambda\bfN_\calO\cap\bfN_\calO)\right)}{
\Eu(T_{\lambda}\Gr_\lambda)
}\right)
d_{w\lambda}\in \Diff_\hbar(T),
\end{equation}
where $W^\lambda$ is the set of minimal length representatives of cosets in $W/W_\lambda$,  $\Eu(-)$ is the $T\times\mathbb
{C}^*_\hbar$-equivariant Euler class, i.e. the product of weights, and $T_{\lambda}\Gr_\lambda
$ is the tangent space of $\Gr_\lambda$ at the torus fixed point $z^\lambda G_\calO/G_\calO$. Note that we have a bijection $W^\lambda\cong W\cdot \lambda$ by $w\mapsto w\lambda$. It is well known that for any $\mu\in W\lambda$, 
\begin{equation}\label{equ:wts}
    T_\mu\Gr_\lambda = \bigoplus_{\alpha\in R, \langle\mu,\alpha\rangle\geq 1}\bigoplus_{n=0}^{\langle\mu,\alpha\rangle - 1}\mathfrak{g}_\alpha z^n,
\end{equation}
where $R$ is the set of roots of $G$.

Since $\bfN$ is a $G$-representation, 
$$w\Eu(z^\lambda \bfN_\calO/(z^\lambda \bfN_\calO\cap \bfN_\calO))
=\Eu(z^{w\lambda}\bfN_\calO/(z^{w\lambda} \bfN_\calO\cap \bfN_\calO)).$$
If $S$ is a one dimensional $T$-invariant summand of $\bfN$ with $T$ weight $\alpha$, we have
\begin{equation}\label{equ:Euquo}
    \Eu(z^{\lambda}S_\calO/(z^{\lambda}S_\calO\cap S_\calO))=\begin{cases}1 , & m\geq 0,\\(\alpha-\tfrac{\hbar}{2})(\alpha-\tfrac{3\hbar}{2})\cdots(\alpha-\tfrac{(2|m|-1)\hbar}{2}), & m< 0,
\end{cases}
\end{equation}
where $m=\langle \lambda, \alpha\rangle$. The factor $\frac{\hbar}{2}$ appears here since $\mathbb{C}_\hbar^\times$ scales $\bfN$ by weight $\frac{1}{2}$.

\subsection{Quasi-minuscule monopole operators}\label{sec:quasi}
Now let us consider the case that $\lambda$ is quasi-minuscule, i.e., $\lambda=\theta^\vee$ with $\theta$ being a highest root of $G$ (when the root system is irreducible, $\theta$ is unique). For any positive root $\alpha\in R^+\setminus\{\theta\}$, $\langle\lambda, \alpha\rangle\in \{0,1\}$, and $0$ is the only element in $X_*(T)^+$ smaller than $\lambda$. Hence, 
\[\overline{\Gr_\lambda}=\Gr_\lambda\sqcup \{e\},\]
where $e\in \Gr_G$ is the identity point. Moreover, $\Gr_\lambda$ is a line bundle over $G/P_\lambda$. By a computation of the tangent weights, 
\[\Gr_\lambda\simeq G\times_{P_\lambda}\mathbb{C}_\theta\]
as $G$-equivariant line bundles over $G/P_\lambda$. Here $\mathbb{C}_\theta$ is the one-dimensional representation of $P_\lambda$. We let $\mathcal{L}$ denote the line bundle on the right hand side. The explicit isomorphism $\mathcal{L}\simeq \Gr_\lambda$ is given by 
\[(g,t)\mapsto gx_\theta(tz)z^\lambda G_\calO/G_\calO,\]
where $x_\theta(t):=\exp(tE_\theta)$ for $E_\theta$ the root vector in $\mathfrak{g}_\theta$.

\begin{example}
\label{exa:limit}
Let us consider the case $G=SL_2$. 
Then $\theta$ is the unique positive root, and 
$$P_\lambda =\left[\begin{matrix}
{*} & 0\\
{*} & {*}
\end{matrix}\right],\qquad G/P_\lambda \cong \mathbb{P}^1,\qquad \mathcal{L}\cong\mathcal{O}(-2).$$
We have 
$$x_\theta(tz)z^\lambda = \left[\begin{matrix}
1 & zt\\0 & 1 
\end{matrix}\right]
\left[\begin{matrix}
z &0 \\0 & z^{-1}
\end{matrix}\right]
=\left[\begin{matrix}
z & t \\ 0& z^{-1}
\end{matrix}\right].
$$
When $t\neq 0$, we have 
$$x_\theta(tz)z^\lambda\left[\begin{matrix}
0 & -t \\ t^{-1} & z 
\end{matrix}\right] = \left[\begin{matrix}
1 & 0 \\ z^{-1}t^{-1} & 1
\end{matrix}\right]
\stackrel{t\to \infty}
\longrightarrow 
\left[\begin{matrix}
1 & 0 \\ 0 & 1
\end{matrix}\right].$$
We get $\lim\limits_{t\rightarrow \infty}gx_\theta(tz)z^\lambda G_\calO=e$ for any $g\in G$.
\end{example}

There is a resolution of $\overline{\Gr_\lambda}$ given by the 
\[p:\mathbb{P}(\calO\oplus\mathcal{L})\rightarrow \overline{\Gr_\lambda},\]
where the fibers at infinity all map to the identity point $e\in \overline{\Gr_\lambda}$, see \cite{NP01}. 

Note that 
    when restricting to $G\times_P \mathbb{C}_\theta$, the fiberwise quotient 
$p^*(\mathcal{T}_{G,\bfN})/p^*(\mathcal{R}_{G,\bfN})$ is a vector bundle of rank 
$r=\dim(z^{\lambda}\bfN_\calO/(z^{\lambda}\bfN_\calO\cap \bfN_\calO))$; 
    at the infinity points, the fiber of $p^*(\mathcal{R}_{G,\bfN})$ equals the fiber of $p^*(\mathcal{T}_{G,\bfN})$.

\begin{lemma}\label{lem:extofR}
There exists a unique subbundle $\tilde{\mathcal{R}}$ of $p^*(\mathcal{T}_{G,\bfN})$ 
such that 
\begin{itemize}
    \item $\tilde{\mathcal{R}}$ coincides with $p^*(\mathcal{R}_{G,\bfN})$ over $p^{-1}(\Gr_\lambda)=G\times_{P_\lambda}\mathbb{C}_\theta$; 
    \item the fiberwise quotient $p^*(\mathcal{T}_{G,\bfN})/\tilde{\mathcal{R}}$
    is a vector bundle of rank $r$. 
\end{itemize}
Moreover, $\tilde{\mathcal{R}}$ is $G$-invariant and is contained in  $p^*(\mathcal{R}_{G,\bfN})$.
\end{lemma}
\begin{proof}
Note that $\mathcal{T}_{G,\bfN}$ (and thus $p^*(\mathcal{T}_{G,\bfN})$) is a sheaf of sublattices of the trivial bundle $\bfN_\calK$.
Let us consider the bundle
$$\mathcal{G}=\left\{
(R,x): 
\begin{matrix}
    \text{$R$ is a sublattice of the fiber of $p^*(\mathcal{T}_{G,\bfN})$}\\
    \text{at $x\in \mathbb{P}(\mathcal{O}\oplus \mathcal{L})$ of codimension $r$}
\end{matrix}\right\}.$$
The bundle $p^*(\mathcal{R}_{G,\bfN}) 
|_{G\times_{P_\lambda}\mathbb{C}_\theta}$ defines a rational section $s$ of $\mathcal{G}$ regular over $G\times_{P_\lambda}\mathbb{C}_\theta$. 
The Lemma asserts that this rational section $s$ can be extended uniquely to a regular section. 
In particular, the uniqueness follows.

Let us denote the fiber of $\mathcal{L}$ (resp. $\mathbb{P}(\mathcal{O}\oplus \mathcal{L})$)
at $1\cdot P_\lambda$ by $\mathbb{A}^1$ (resp., $\mathbb{P}^1$). 
Note that all the above construction is $G$-equivariant, 
in particular, $\mathcal{G}=G\times_{P_\lambda} (\mathcal{G}|_{\mathbb{P}^1})$ and the rational section $s$ can be reconstructed from $s|_{\mathbb{A}^1}$. 
As a result, it suffices to extend $s$ from $\mathbb{A}^1$ to $\mathbb{P}^1$. 
This follows from the fact that $\mathcal{G}\to \mathbb{P}(\mathcal{O}\oplus \mathcal{L})$ is ind-projective. 

Actually, $\mathcal{G}$ can be approximated by closed subvarieties of Grassmannian bundle 
$$\mathcal{G}_{N}
=\left\{(V,x): 
\begin{matrix}
\text{$V$ is a $z$-invariant subspace of the fiber}\\
\text{of $p^*(\mathcal{T}_{G,\bfN})/
p^*(z^N \mathcal{T}_{G,\bfN})$ at $x\in \mathbb{P}(\mathcal{O}\oplus \mathcal{L})$ of codimension $r$}
\end{matrix}\right\}.$$

Any rational section $s$ is valued in some $\mathcal{G}_N$ for $N\gg 0$. 
\end{proof}

\begin{remark}
The Lemma confirms the question asked in \cite[Remark A.14]{BFN19}. 

Note that here our proof is specific for the quasi-miniscule case. 
For more general $\lambda$, by \cite[Lemma 1.2]{CCL25}, a resolution of $\overline{\Gr_\lambda}$ exists such that a property similar to Lemma \ref{lem:extofR} holds, but the resolution is not explicit and depends on $\mathcal{R}_{G,\bfN}$.  
\end{remark}

We will write $\tilde{\mathcal{R}}=\tilde{\mathcal{R}}_{G,\bfN,\leq \lambda}$ to indicate the dependence of $G$ and $\bfN$. 
By our Lemma \ref{lem:extofR}, 
the map $p$ also induces a proper map $\tilde{p}:\tilde{\calR}_{G,\bfN,\leq \lambda}\rightarrow \mathcal{T}_{G,\bfN}$, with image lying in $\mathcal{R}_{G,\bfN,\leq \lambda}:=\pi^{-1}(\overline{\Gr_\lambda})\subset \calR_{G,\bfN}$. 
We have the following diagram
$$\xymatrix@!C=4pc{
p^*(\mathcal{T}_{G,\bfN})\ar[rr]\ar[dd]
&& \mathcal{T}_{G,\bfN,\leq \lambda}\ar[dd]|{\hole}\ar[rr]
&& \mathcal{T}_{G,\bfN}\ar[dd]|{\hole}\\
& \tilde{\mathcal{R}}_{G,\bfN,\leq \lambda}\ar[ul]\ar[dl]
\ar[rr]^{\tilde{p}\qquad\quad } && \mathcal{R}_{G,\bfN,\leq \lambda}\ar[ul]\ar[dl]\ar[rr]&&
\mathcal{R}_{G,\bfN}.\ar[ul]\ar[dl]\\
\mathbb{P}(\mathcal{O}\oplus \mathcal{L})
\ar[rr]^-{p} &&\overline{\Gr_\lambda}\ar[rr]&&\Gr_G}$$

For any $f\in H_G^*(\Gr_\lambda)\simeq \mathbb{C}[\frak{t}]^{W_\lambda}$, we can pull it back to $\tilde{\calR}_{G,\bfN,\leq \lambda}$ via the projection map $\tilde{\calR}_{G,\bfN,\leq \lambda}\rightarrow\mathbb{P}(\calO\oplus\mathcal{L})\rightarrow G/P_\lambda$. Therefore, we can consider the class
\[\tilde{p}_*(f[\tilde{\calR}_{G,\bfN,\leq \lambda}])\in \mathcal{A}_\hbar(G,\bfN), \]
and we have the following formula.

\begin{theorem}\label{thm:quasi}
    With the above notations,
    \begin{align*}
    &z^*\circ (\iota_*)^{-1}\big(\tilde{p}_*(f[\tilde{\calR}_{G,\bfN,\leq \lambda}])\big)\\
    =&\sum_{w\in W^\lambda}w\left(f\cdot \frac{\Eu\left(z^\lambda\bfN_\calO/(z^\lambda\bfN_\calO\cap\bfN_\calO)\right)}{\Eu(T_{\lambda}\Gr_\lambda)}\right)(d_{w\lambda}-1)+\cdots\in \Diff_\hbar(T),
\end{align*}
where $\cdots$ denotes some polynomial in $H_{G\times \mathbb{C}^\times_\hbar}^*(\pt)$. In particular, 
\[\sum_{w\in W^\lambda}w\left(f\cdot \frac{\Eu\left(z^\lambda\bfN_\calO/(z^\lambda\bfN_\calO\cap\bfN_\calO)\right)}{\Eu(T_{\lambda}\Gr_\lambda)}\right)(d_{w\lambda}-1)\in \mathcal{A}_\hbar(G,\bfN).\]
\end{theorem}
\begin{proof}
    The idea of the proof is similar to the one in \cite[Theorem A.13]{BFN19}. The torus fixed points in $\tilde{\calR}$ are $\{(w\cdot P_\lambda, 0), (w\cdot P_\lambda, \infty)\mid w\in W^\lambda\}$. 

    Since everything is $G$-equivariant, we only need to compute the localization contribution when $w\in W^\lambda$ is the identity element. By definition, 
    \[\begin{aligned}
    \Eu(T_{(1\cdot P_\lambda, 0)}\mathbb{P}(\calO\oplus\mathcal{L}))
    & =(\theta+\hbar)\cdot \Eu(T_{1\cdot P_\lambda}G/P_\lambda)
    =\Eu(T_\lambda \Gr_\lambda),\\
    \Eu(T_{(1\cdot P_\lambda, \infty)}\mathbb{P}(\calO\oplus\mathcal{L}))
    & =(-\theta-\hbar)\cdot \Eu(T_{1\cdot P_\lambda}G/P_\lambda)
    =-\Eu(T_\lambda \Gr_\lambda).
    \end{aligned}\]
    Hence, by localization we get
    \begin{align*}
        &z^*\circ (\iota_*)^{-1}\big(\tilde{p}_*(f[\tilde{\calR}])\big)\\
        =&\sum_{w\in W^\lambda}w\bigg(f\frac{\Eu\big((p^*(\mathcal{T}_{G,\bfN})/\tilde{\calR})_{(1\cdot P_\lambda, 0)}\big)}{\Eu(T_{(\lambda, 0)}\mathbb{P}(\calO\oplus\mathcal{L}))}\bigg)d_{w\lambda}+\sum_{w\in W^\lambda}w\bigg(f\frac{\Eu\big((p^*(\mathcal{T}_{G,\bfN})/\tilde{\calR})_{(1\cdot P_\lambda, \infty)}\big)}{\Eu(T_{(1\cdot P_\lambda,\infty)}\mathbb{P}(\calO\oplus\mathcal{L}))}\bigg)\\
        =&\sum_{w\in W^\lambda}w\bigg(f\frac{\Eu\big(z^\lambda\bfN_\calO/(z^\lambda\bfN_\calO\cap\bfN_\calO)\big)}{\Eu(T_\lambda\Gr_\lambda)}\bigg)d_{w\lambda}-\sum_{w\in W^\lambda}w\bigg(f\frac{\Eu\big((p^*(\mathcal{T}_{G,\bfN})/\tilde{\calR})_{(1\cdot P_\lambda, \infty)}\big)}{\Eu(T_\lambda\Gr_\lambda)}\bigg).
    \end{align*}
Suppose the $T\times \mathbb{C}^\times_\hbar$-weights in $z^\lambda\bfN_\calO/(z^\lambda\bfN_\calO\cap\bfN_\calO)$ are $\{\gamma_1,\cdots, \gamma_d\}$. 
The fiber of $\mathbb{P}(\mathcal{O}\oplus \mathcal{L})$ over $1\cdot P_\lambda$ is isomorphic to $\mathbb{P}^1$. Restricting the vector bundle $p^*(\mathcal{T}_{G,\bfN})/\tilde{\mathcal{R}}$ to this fiber and applying Lemma \ref{lem:BirkhoffGrothendieckTh} below, we obtain that
the $T\times \mathbb{C}^\times_\hbar$-weights in $\big(p^*(\mathcal{T}_{G,\bfN})/\tilde{\calR}\big)_{(1\cdot P_\lambda, \infty)}$ can be ordered as $\gamma_1',\ldots,\gamma_d'$ with $\gamma_i'-\gamma_i\in \mathbb{Z}(\theta+\hbar)$. Therefore, 
\[F:=f\frac{\Eu\big(z^\lambda\bfN_\calO/(z^\lambda\bfN_\calO\cap\bfN_\calO)\big)}{\theta+\hbar}-f\frac{\Eu\big((p^*(\mathcal{T}_{G,\bfN})/\tilde{\calR})_{(1\cdot P_\lambda, \infty)}\big)}{\theta+\hbar}\in H_T^*(\pt)[\hbar].\]
Moreover, they are $W_\lambda$ invariant as $f, \Eu\big(z^\lambda\bfN_\calO/(z^\lambda\bfN_\calO\cap\bfN_\calO)\big),\Eu\big((p^*(\mathcal{T}_{G,\bfN})/\tilde{\calR})_{(1\cdot P_\lambda, \infty)}\big)$ and $\theta+\hbar$ are all $W_\lambda$-invariant.
Hence, by the localization theorem on $G/P_\lambda$, we get that
\begin{align*}
    &\sum_{w\in W^\lambda}w\bigg(f\frac{\Eu\big(z^\lambda\bfN_\calO/(z^\lambda\bfN_\calO\cap\bfN_\calO)\big)}{\Eu(T_\lambda\Gr_\lambda)}\bigg)-\sum_{w\in W^\lambda}w\bigg(f\frac{\Eu\big((p^*(\mathcal{T}_{G,\bfN})/\tilde{\calR})_{(1\cdot P_\lambda, \infty)}\big)}{\Eu(T_\lambda\Gr_\lambda)}\bigg)\\
    =&\sum_{w\in W^\lambda}w\bigg(\frac{F}{\Eu(T_{1\cdot P_\lambda}G/P_\lambda)}\bigg)\in H_{G\times \mathbb{C}^\times_\hbar}^*(\pt).
\end{align*}
This finishes the proof of the theorem.
\end{proof}

\begin{lemma}\label{lem:BirkhoffGrothendieckTh}
Let $\mathbb{T}$ be a torus acting on $\mathbb{P}^1$ with fixed points $0,\infty$.
Let $\mathcal{V}$ be a $\mathbb{T}$-equivariant vector bundle. 
Then there exists a bijection $\psi$ between the weights of the fiber of $\mathcal{V}$ at $0$ and at $\infty$ such that 
$$\psi(\gamma)-\gamma\in \mathbb{Z}\tau$$
where $\tau$ is the $\mathbb{T}$-weight of $T_0\mathbb{P}^1$. 
\end{lemma}
\begin{proof}
This follows from the $\mathbb{T}$-equivariant Birkhoff--Grothendieck Theorem. Since we did not find a reference, we sketch a proof here for readers' convenience.

Let $d$ be the maximal integer such that $H^0(\mathbb{P}^1,\mathcal{V}(-d))\neq 0$. Fix a $\mathbb{T}$-equivariant structure on $\mathcal{O}(d)$. Then the space $H^0(\mathbb{P}^1,\mathcal{V}(-d))$ is a finite-dimensional $\mathbb{T}$-representation. Choose a non-zero weight vector $s$ in it. After modifying the $\mathbb{T}$-equivariant structure on $\mathcal{O}(d)$ by some $\mathbb{T}$-character, $s$ defines a $\mathbb{T}$-equivariant morphism
\[s:\mathcal{O}(d)\rightarrow \mathcal{V}.\]
By the maximality of $d$, $s$ is a nowhere vanishing morphism. Hence, its cokernel $\mathcal{Q}$ is a vector bundle, and we get a $\mathbb{T}$-equivariant short exact sequence 
\[0\rightarrow \mathcal{O}(d)\rightarrow \mathcal{V}\rightarrow \mathcal{Q}\rightarrow 0.\]

By the usual non-equivariant Birkhoff--Grothendieck Theorem, this sequence splits non-equivariantly. Hence,
$\Hom(\mathcal{Q},\mathcal{V})\rightarrow\Hom(\mathcal{Q},\mathcal{Q})$ is surjective. Taking $T$-invariants, we still get a surjective map. Hence, we get a $\mathbb{T}$-equivariant splitting.

Thus the Lemma follows from the case when $\mathcal{V}$ is a line bundle. Then we can assume 
$\mathcal{V}=\mathcal{I}_{0}^m\otimes \mathbb{C}_\chi$ for some integer $m\in \mathbb{Z}$ and some $\mathbb{T}$-character $\chi$. Here $\mathcal{I}_{0}$ is the ideal sheaf of $\{0\}$. Note that 
$$\begin{aligned}
\Eu(\mathcal{V}_0) & = -m \tau +\chi,
& \qquad 
\Eu(\mathcal{V}_\infty) & = \chi\\
\end{aligned}.$$
Clearly, $-m\tau+\chi-\chi\in \mathbb{Z}\cdot \tau$. 
\end{proof}

\begin{figure}[hhhhhhhhh]
    \centering
$$
\begin{matrix}
\begin{tikzpicture}[scale=0.75]
\clip (-4,4.7) rectangle (4,-4.7);
\foreach \i in {-9,...,6}
    \draw[line width=7,yellow!50!white] (-6,\i-1.5) to (6,\i+4.5);
\draw[->] (-4,0) to (4,0); 
\draw[->] (0,-4.5) to (0,4.5) node [right] {\(z\)}; 
\foreach \i in {-3,-1,1,3}
    \foreach \j in {-4,...,4}
        \node (my\i\j) at (\i,\j) {%
            \ifnum\i>\j%
                \color{lightgray}
            \else 
                \ifnum\j<0
                    \color{red}
                \else
                    \color{black}
            \fi\fi$\bullet$};
\draw[ultra thick,rounded corners] (-3.3,4.5) to (-3.3,-0.3) to (0,-0.3) to (3.3,3.0) to (3.3, 4.5);
\draw[ultra thick,red,rounded corners] (-3.6,4.5) to (-3.6,-3.6) to (-3.0,-3.6) to  (3.6,3.0) to (3.6, 4.5);
\end{tikzpicture}\\
\textnormal{the fiber at $(1\cdot P_\lambda,0)$}
\end{matrix}
\qquad 
\begin{matrix}
\begin{tikzpicture}[scale=0.75]
\clip (-4,4.7) rectangle (4,-4.7);
\foreach \i in {-9,...,6}
    \draw[line width=7,yellow!50!white] (-6,\i-1.5) to (6,\i+4.5);
\draw[->] (-4,0) to (4,0); 
\draw[->] (0,-4.5) to (0,4.5) node [right] {\(z\)}; 
\foreach \i in {-3,-1,1,3}
    \foreach \j in {-4,...,4}
        \node (my\i\j) at (\i,\j) {%
            \ifnum\j<0%
                \color{lightgray}
            \else 
                \ifnum\the\numexpr4*\j-\i-2\relax<0
                    \color{red}
                \else
                    \color{black}
            \fi\fi$\bullet$};
\draw[ultra thick,rounded corners] (-3.3,4.5) to (-3.3,-0.3) to (3.3,1.35)  to (3.3,3.0) to (3.3, 4.5);
\draw[ultra thick,red,rounded corners] (-3.6,4.5) to (-3.6,-0.5) to (3.6,-0.5) to (3.6,4.5);
\end{tikzpicture}\\
\textnormal{the fiber at $(1\cdot P_\lambda,\infty)$}
\end{matrix}
$$
\caption{Computation of the limit}
    \label{fig:limit}
\end{figure}
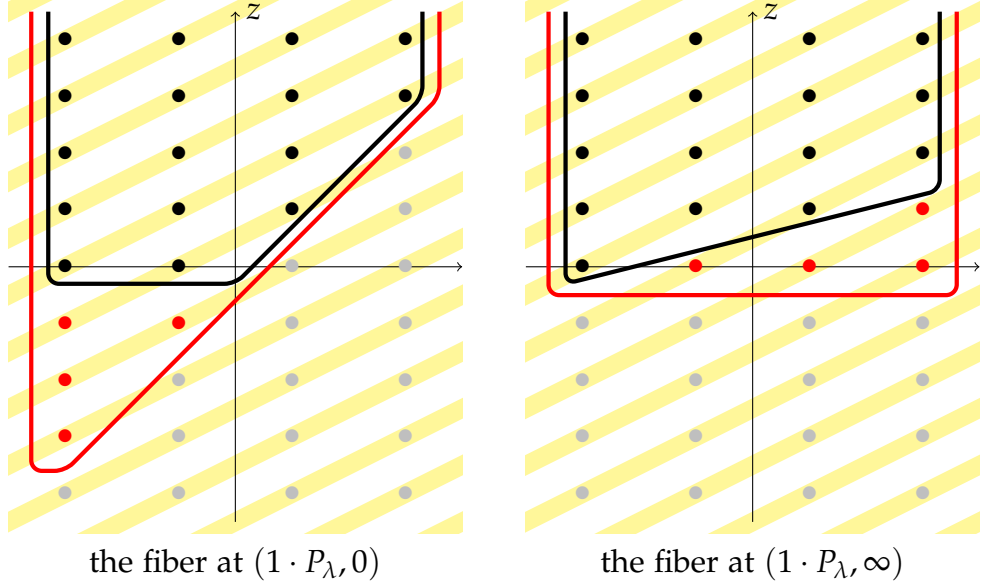

\begin{example}
Consider the case $\bfN=\mathfrak{g}$ is the adjoint representation of $G$ and take $f=1$, we get 
$$\sum_{w\in W^\lambda}
w\left(\frac{-\theta-\tfrac{3\hbar}{2}}{\theta+\hbar}
\prod_{\langle\theta^\vee,\alpha\rangle>0}\frac{-\alpha-\tfrac{\hbar}{2}}{\alpha}
\right)(d_{w\lambda}-1)\in \mathcal{A}_\hbar(G,\bfN).$$
This agrees with \cite[Section 6]{macdonald2000orthogonal}; see also \cite[Theorem A.13]{BFN19}. 
\end{example}

\begin{example}
Let us consider the case $G=SL_2$, and $$\bfN = \operatorname{span}(v,Ev,E^2v,E^3v)$$
the four-dimensional irreducible representation with lowest weight vector $v$. Then $\lambda=\alpha^\vee$, with $\alpha$ being the positive root. The lowest weight vector $v$ has torus weight $-\frac{3\alpha}{2}$.
By definition, 
the fiber of $p^*(\mathcal{R}_{G,\bfN})$ at $(1\cdot P_\lambda,0)$ is 
$$
R:=\operatorname{span}(z^{i}v:i\geq 0)\oplus 
\operatorname{span}(z^{i}Ev:i\geq 0)\oplus 
\operatorname{span}(z^{i}E^2v:i\geq 1)\oplus 
\operatorname{span}(z^{i}E^3v:i\geq 3).$$
The limit, i.e. the fiber of $\tilde{\mathcal{R}}_{G,\bfN,\leq\lambda}$ at $(1\cdot P_\lambda,\infty)$ is 
$$
\operatorname{span}(z^{i}v:i\geq 0)\oplus 
\operatorname{span}(z^{i}Ev:i\geq 1)\oplus 
\operatorname{span}(z^{i}E^2v:i\geq 1)\oplus 
\operatorname{span}(z^{i}E^3v:i\geq 2).$$
The computation can be illustrated in Figure \ref{fig:limit}, where 
\begin{itemize}
    \item the horizontal axis corresponds to the $T$-weights;
    \item the vertical axis corresponds to the $\mathbb{C}^\times_\hbar$-weights; 
    \item 
the black dots correspond to the weights of $\tilde{\mathcal{R}}_{G,\bfN,\leq\lambda}$; 
    \item 
the red dots correspond to the weights of $p^*(\mathcal{T}_{G,\bfN})/\tilde{\mathcal{R}}_{G,\bfN,\leq\lambda}$;
    \item the yellow strips correspond to the action of $x_\theta(tz)$. 
\end{itemize}
Let us explain why $z^2E^3v$ and $zE^2v$ are contained in the limit as an example. 

For $t\neq 0$, we have 
$$\begin{aligned}
x_{\theta}(tz)zE^2v & = zE^2 v + tz^2 E^3v
= t\cdot \big(\tfrac{1}{t}zE^2v+z^2E^3v\big)\\
x_{\theta}(tz)Ev & = Ev+tzE^2v +\frac{(tz)^2}{2}E^3v
= t\cdot \big(\tfrac{1}{t}Ev+\tfrac{z}{2}E^2v
+\tfrac{1}{2}x_{\theta}(tz)zE^2v\big).
\end{aligned}$$
Thus 
$\tfrac{1}{t}zE^2v+z^2E^3v$ and $\tfrac{1}{t}Ev+\tfrac{z}{2}E^2v$
are contained in the fiber of $p^*(\mathcal{R}_{G,\bfN})$ at $(1\cdot P_\lambda,t)$. 
As a result, their limits 
$$z^2E^3v=\lim_{t\to\infty} \tfrac{1}{t}zE^2v+z^2E^3v,\qquad 
\tfrac{z}{2}E^2v= \lim_{t\to\infty}\tfrac{1}{t}Ev+\tfrac{z}{2}E^2v$$
are contained in the fiber of $\tilde{\mathcal{R}}$ at $(1\cdot P_\lambda,\infty)$.
\end{example}

\subsection{More properties of the Coulomb branch}
The quantized Coulomb branch algebra of the cotangent type only depends on $\bfN\oplus \bfN^*$. Precisely, assume we have two choices of $\bfN$
\begin{equation}\label{eq:twodecomp}
\bfN_1\oplus \bfN_1^*\cong \bfN_2\oplus \bfN_2^*.     
\end{equation}

Let $\sigma$ be the automorphism 
\[\sigma:\Diff_\hbar(T)\rightarrow \Diff_\hbar(T)\]
which is the identity on $\mathbb{C}(\mathfrak{t}\times\mathbb{A}^1)$ and for any cocharacter $\lambda$ of $T$, 
\begin{equation}\label{equ:sigma}  \sigma\bigg(\Eu(z^\lambda\bfN_{2,\calO}/(z^\lambda\bfN_{2,\calO}\cap\bfN_{2,\calO})d_\lambda\bigg)=\pm
\Eu(z^\lambda\bfN_{1,\calO}/(z^\lambda\bfN_{1,\calO}\cap\bfN_{1,\calO}))d_\lambda,
\end{equation}
where
$$\pm = 
(-1)^{\sum_\chi\max(\chi(\lambda),0)\dim \underline{\bfN}(\chi)}$$
for a $T$-complement $\underline{\bfN}$ of $\bfN_1\cap \bfN_2$ in $\bfN_2$. 
Then $\sigma$ restricts to an isomorphism between $\mathcal{A}_\hbar(G,\bfN_2)$ and $\mathcal{A}_\hbar(G,\bfN_1)$, see \cite[Section 6(viii)]{BFN18}.

Let $F$ be a reductive group and $\tilde{G}=G\times F$. 
We assume further that the $G$-representation $\bfN$ can be extended to a $\tilde{G}$-representation. 
Then we can slightly extend the above definition by 
$$\mathcal{A}_\hbar^F(G,\bfN) := 
H^{(G_{\calO}\times F_{\calO})\rtimes\mathbb{C}_\hbar^\times}_*(\calR_{G,\bfN}).$$

Then the embedding \eqref{eq:diff} becomes
$$\mathcal{A}_\hbar^F(G, \bfN)\hookrightarrow H_F^*(\pt)\otimes\Diff_\hbar(T),$$
and the formulas in \eqref{eq:minu} and Theorem \ref{thm:quasi} still hold.

\section{Coulomb branches of separated cotangent type}\label{sec:noncotCB}

In this section, we propose a definition of the quantized Coulomb branch algebra of noncotangent type in the special case called \emph{separated cotangent type}. 
We prove it is compatible with the definition \cite{BDFRT} of (commutative) Coulomb branch algebra. The results in this section also work if there is a flavor symmetry group $F$.

\subsection{Separated Cotangent Coulomb branch}
Let $\bfM$ be a symplectic representation of $G$, i.e. $\rho:G\to \Sp(\bfM)$.

We say $(G,\bfM)$ is of \emph{separated cotangent type} if 
$$G=G_1\times G_2$$
such that $(G_i,\bfM)$ is of cotangent type for $i=1,2$.

Let us pick a maximal torus $T_i$ of $G_i$, and denote 
$$T=T_1\times T_2,\qquad 
TG_1=G_1\times T_2,\qquad 
TG_2=T_1\times G_2.$$

\begin{lemma}
For $i=1,2$, $(TG_i,\bfM)$ is of cotangent type. 
\end{lemma}
\begin{proof}
Let us prove for $i=1$. 
Note that the choices of $\bfN_i$ form a nonempty projective variety. 
Explicitly, it is given by the $G_1$-invariant locus of Lagrangian Grassmannian $LG(\bfM)$ of $\bfM$.

Since the action of $T_2$ commutes with $G_1$ and preserves the symplectic form, it acts on the variety. 
In particular, it must have a $T_2$-fixed point. 
\end{proof}

So we can assume $\bfM = \bfN_i\oplus \bfN_i^*$ for some $TG_i$-representation $\bfN_i$ for $i=1,2$. 
We have the following commutative diagram
$$\xymatrix{
\mathcal{A}_\hbar(TG_2,\bfN_2)
\ar@{<-}[r]^{\iota_{2,*}}& 
\mathcal{A}_\hbar(T,\bfN_2)^W
\ar[r]^-{\subset}& 
\mathcal{A}_\hbar(T,\bfN_2)\ar[d]^{\wr}
\ar[r]^-{z^*}& 
\mathcal{A}_\hbar(T,0)
\ar[r]^-{\subset}& 
\Diff_\hbar(T)\ar[d]^{\wr}_{\sigma}
\\
\mathcal{A}_\hbar(TG_1,\bfN_1)
\ar@{<-}[r]^{\iota_{1,*}}& 
\mathcal{A}_\hbar(T,\bfN_1)^W
\ar[r]^-{\subset}& 
\mathcal{A}_\hbar(T,\bfN_1)
\ar[r]^-{z^*}& 
\mathcal{A}_\hbar(T,0)
\ar[r]^-{\subset}& 
\Diff_\hbar(T),
}$$
where $\sigma$ is the automorphism defined in \eqref{equ:sigma}. 
Let us identify $\mathcal{A}_{\hbar}(TG_i,\bfN_i)$ as subalgebras of $\mathsf{Diff}_\hbar(T)$ via 
\begin{equation}\label{eq:idenassubalgofDiff}
\begin{aligned}
z^*\circ (\iota_{1,*})^{-1}:& 
\mathcal{A}_{\hbar}(TG_1,\bfN_1)
\stackrel{\subset}\longrightarrow 
\mathsf{Diff}_\hbar(T),\\
\sigma\circ z^*\circ (\iota_{2,*})^{-1}:& 
\mathcal{A}_{\hbar}(TG_2,\bfN_2)
\stackrel{\subset}\longrightarrow \mathsf{Diff}_\hbar(T).
\end{aligned}
\end{equation}

Motivated by Theorem \ref{thm:compwithBDFRT} below, we give the following definition.
\begin{definition}\label{def:noncot}
We define the quantized Coulomb branch algebra
$\mathscr{A}_\hbar(G,\bfM)\subset \Diff_\hbar(T)$ to be the intersection of subalgebras
$$
\mathcal{A}_\hbar(TG_1,\bfN_1)_{\operatorname{Frac}H_{T_2}^*(\pt)}
\bigcap
\mathcal{A}_\hbar(TG_2,\bfN_2)_{\operatorname{Frac}H_{T_1}^*(\pt)}.
$$
Here $\operatorname{Frac}H_{T_i}^*(\pt)$ denotes the fraction field of $H_{T_i}^*(\pt)$, and for any subalgebra $\mathcal{A}$ of  $\Diff_\hbar(T)$, we use $\mathcal{A}_{\operatorname{Frac}H_{T_i}^*(\pt)}$ to denote the subalgebra of $\Diff_\hbar(T)$ generated by  $\mathcal{A}$ and $\operatorname{Frac}H_{T_i}^*(\pt)$.
\end{definition}
\begin{remark}
    In this paper, we use the notation $\mathscr{A}_\hbar(G,\bfM)$ to denote the quantized Coulomb branch algebra of non-cotangent type $(G,\bfM)$, while we use $\mathcal{A}_\hbar(G,\bfN)$ to denote the quantized Coulomb branch algebra of cotangent type $(G,\bfN\oplus\bfN^*)$. 
\end{remark}

\begin{remark}
Up to a noncanonical isomorphism, the algebra in Definition~\ref{def:noncot}
is independent of the choices of the maximal tori $T_1,T_2$ and of the compatible Lagrangians $\bfN_1,\bfN_2$. 

Note that maximal tori are all conjugate, thus it suffices to show the independence for Lagranians. 
First we warn readers that the following diagram 
$$
\xymatrix{
\mathcal{A}_\hbar(T,\bfN_2)\ar[r]^{\sim}
\ar[d]^{\wr}
&
\mathcal{A}_\hbar(T,\bfN'_2)\ar[d]^{\wr}\\
\mathcal{A}_\hbar(T,\bfN_1)\ar[r]^{\sim}& \mathcal{A}_\hbar(T,\bfN_1')}
$$
does NOT necessarily commute, if all the arrows are changing Lagrangians, i.e. the restrictions of $\sigma$ in \eqref{equ:sigma}. 

In order to get a commutative diagram, we need to twist further by a group homomorphism $\varphi: X_*(T)=X_*(T_1)\oplus X_*(T_2)\to \{\pm 1\}$, i.e., $d_\lambda\mapsto \varphi(\lambda)d_\lambda$. 
Now the component $\varphi_i:X_*(T_i)\to \{\pm 1\}$ can be lifted to an automorphism of 
$\mathcal{A}_\hbar(TG_{3-i},\bfN_{3-i})$. 
Thus the above diagram commutes after modifying the horizontal arrows. With this, it is easy to see that $\mathscr{A}_\hbar(G,\bfM)$ does not depend on the choices of the Lagrangians. 
\end{remark}

We equip $\Diff_\hbar(T)$ with the usual $W_1$ action, and 
a new action of  $W_2$ by $w(X):=\sigma(w\cdot_{\rm usual}(\sigma^{-1}(X)))$ for $w\in W_2$. 

\begin{lemma}
The new action of $W_2$ commutes with the action of $W_1$. 
\end{lemma}
\begin{proof}
For $w=w_1\in W_1$, we have 
$$w\big(f(t,\hbar)d_\lambda\big)=(wf)(t,\hbar)d_{w\lambda},$$
in particular, the action is characterized by 
$$
\begin{aligned}
&  w \left(f(t,\hbar)\cdot \Eu(z^\lambda\bfN_{1,\calO}/(z^\lambda\bfN_{1,\calO}\cap\bfN_{1,\calO}))d_\lambda
\right)\\
=\,\,& (wf)(t,\hbar)\cdot 
\Eu(z^{w\lambda}\bfN_{1,\calO}/(z^{w\lambda}\bfN_{1,\calO}\cap\bfN_{1,\calO}))d_{w\lambda}
\end{aligned}$$
since $\bfN_1$ is invariant under $w$. 

For $w=w_2\in W_2$, we similarly have 
\begin{align}\label{equ:twistedw2}
&  w\left(f(t,\hbar)\cdot \Eu(z^\lambda\bfN_{1,\calO}/(z^\lambda\bfN_{1,\calO}\cap\bfN_{1,\calO}))d_\lambda
\right)\\
=\,\,&  \pm  (wf)(t,\hbar)\cdot \Eu(z^{w\lambda}\bfN_{1,\calO}/(z^{w\lambda}\bfN_{1,\calO}\cap\bfN_{1,\calO}))d_{w\lambda},\notag
\end{align}
where
$$\pm = (-1)^{a(\lambda)-a(w_2\lambda)},\qquad a(\lambda):=
\sum_\chi\max(\chi(\lambda),0)
\dim \underline{\bfN}(\chi)\bmod 2.$$
Let $w_i\in W_i$ for $i=1,2$. Since the usual actions of $w_1$ and $w_2$ commute, the action of $g=w_1^{-1}w_2^{-1}w_1w_2$ is given by 
$$g(f(t,\hbar)d_\lambda) = (-1)^{a(\lambda)+a(w_1\lambda)+a(w_2\lambda)+a(w_1w_2\lambda)}f(t,\hbar)d_\lambda.$$
Note that $g$ is an algebra homomorphism, and 
$$
g(d_{(\lambda_1,0)}) = d_{(\lambda_1,0)},\qquad 
g(d_{(0,\lambda_2)}) = d_{(0,\lambda_2)}$$
for $\lambda_i\in X_*(T_i)$. 
We conclude that $g$ is the identity, i.e. $w_1w_2=w_2w_1$. 
\end{proof}

We also need the following result.
\begin{theorem}\label{thm:noncotsub}
Under the above identification, for $i=1,2$ we have 
    $$
\mathcal{A}_{\hbar}(TG_i,\bfN_i)
\cap \Diff_\hbar(T)^{W}\subset
\mathscr{A}_{\hbar}(G,\bfM).$$
\end{theorem}
\begin{proof}
We prove the case for $i=1$, as the other case is symmetric. First of all,
\[\mathcal{A}_\hbar(TG_1,\bfN_1)
\cap \Diff_\hbar(T)^{W}
\subset
\mathcal{A}_\hbar(TG_1,\bfN_1)
\subset 
\mathcal{A}_\hbar(TG_1,\bfN_1)_{\operatorname{Frac}H_{T_2}^*(\pt)}.\]
By \cite[Lemma 5.9]{BFN18}, we have
\[\mathcal{A}_\hbar(TG_1,\bfN_1)_{\operatorname{Frac} H_{T_1}^*(\pt)}=\mathcal{A}_\hbar(T,\bfN_1)_{\operatorname{Frac} H_{T_1}^*(\pt)}.\]
Therefore,
\begin{align*}
    \mathcal{A}_\hbar(TG_1,\bfN_1)
\cap \Diff_\hbar(T)^{W}
&
\subset \mathcal{A}_\hbar(TG_1,\bfN_1)_{\operatorname{Frac} H_{T_1}^*(\pt)}
\cap \Diff_\hbar(T)^{W}\\
& = \mathcal{A}_\hbar(T,\bfN_1)_{\operatorname{Frac} H_{T_1}^*(\pt)}
\cap \Diff_\hbar(T)^{W}.
\end{align*}
Let $x\in \mathcal{A}_\hbar(TG_1,\bfN_1)
\cap \Diff_\hbar(T)^{W}$. Since $\sigma(\mathcal{A}_\hbar(T,\bfN_2))=\mathcal{A}_\hbar(T,\bfN_1)$, and $x\in \Diff_\hbar(T)$ is $W_2$ invariant, we get that $\sigma^{-1}(x)\in \mathcal{A}_\hbar(T,\bfN_2)_{\operatorname{Frac} H_{T_1}^*(\pt)}$ is $W_2$ invariant under the usual action. Hence, by \cite[Lemma 5.9 and 5.10]{BFN18}, $\sigma^{-1}(x)\in \mathcal{A}_\hbar(TG_2,\bfN_2)_{\operatorname{Frac} H_{T_1}^*(\pt)}$. Therefore, 
\[x\in \mathcal{A}_\hbar(TG_2,\bfN_2)_{\operatorname{Frac} H_{T_1}^*(\pt)}\cap \Diff_\hbar(T)^{W}.\]
This finishes the proof.
\end{proof}

Let $F$ be a reductive group. Assume the $G$-representation $\bfM$ can be extended to a ($G\times F$)-representation and the Lagrangian $\bfN_i$ is chosen to be $F$-invariant. We can similarly define 
$\mathscr{A}_\hbar^F(G,\bfM)\subset H_F^*(\pt)\otimes \Diff_\hbar(T)$ as an intersection, and Theorem \ref{thm:noncotsub} still holds.

\subsection{Review of \cite{BDFRT}}
Let $\rho: G\to \Sp(\bfM)$ be a symplectic representation. 
We say $(G,\bfM)$ satisfies the anomaly cancellation condition if the induced map between the fourth homotopy groups $\rho:\pi_4(G)\to \pi_4(\Sp(\bfM))\cong \{\pm 1\}$ is trivial.

Let $\Bil(G)$ be the abelian group of symmetric $G$-invariant bilinear forms $B:S^2\mathfrak{g}\to \mathbb{C}$ such that 
\begin{itemize}
    \item 
$B(\lambda,\mu)\in \mathbb{Z}$ for all cocharacters $\lambda,\mu$ and
    \item 
$B(\nu,\nu)\in 2\mathbb{Z}$ for all coroots $\nu$. 
\end{itemize}
By {\cite[Proposition 4.1.2]{BDFRT}}, 
the symplectic representation $\rho$ satisfies the anomaly cancellation condition if and only if 
\begin{equation}\label{equ:anomaly}
    \operatorname{Tr}(d\rho(-),d\rho(-))
\in 2\Bil(G).
\end{equation}

Assume $(G,\bfM)$ satisfies the anomaly cancellation condition in \cite{BDFRT}. 
The \emph{Coulomb branch algebra} is defined to be 
$$
\mathscr{A}_{G,\bfM} 
= H^*_{G_\calO}(\Gr_G,\mathcal{R})$$
for some ring object $\mathcal{R}\in D_{G_\calO}(\Gr_G)$. 
The ring object $\mathcal{R}$ is determined by the choice of a super factorization square root $\sqrt{\mathcal{L}}$ over $\Gr_G$ for $\mathcal{L}=\mathcal{L}_{(G,\bfM)}=\rho^*\det$ where
$\rho:\Gr_G\to \Gr_{\Sp(\bfM)}$ is the natural map induced by $\rho:G\to \Sp(\bfM)$ and $\det$ is the determinant line bundle over $\Gr_{\Sp(\bfM)}$. 
By \cite[Prop 4.1.2]{BDFRT}, the anomaly cancellation condition is equivalent to the existence of $\sqrt{\mathcal{L}}$. Actually, by \cite[Section 2]{BDFRT}, the Picard group of isomorphism classes of super factorization line bundles 
$$\operatorname{SPic}_{\rm fact}(\Gr_G)\cong 
\Bil(G),$$
with $\mathcal{L}$ corresponding to $\rho^*\operatorname{Tr}=\operatorname{Tr}(d\rho(-),d\rho(-))$. 
Moreover, by \cite[Section 4.1]{BDFRT}, the square root $\sqrt{\mathcal{L}}$ is unique, up to a non-unique isomorphism, since 
$$\operatorname{Aut}(\sqrt{\mathcal{L}}) =
\operatorname{Hom}(\pi_1(G),\{\pm 1\}).$$
We remark that here the automorphism group is taken with respect to the factorization structure over $\sqrt{\mathcal{L}}$, not only as a line bundle.

In particular, the Coulomb branch algebra $\mathscr{A}_{G,\bfM}$ is defined up to an automorphism of
$$
\operatorname{Hom}(\pi_1(G),\{\pm 1\}).$$
Explicitly, the connected components of $\Gr_G$ can be naturally parametrized by $\pi_1(G)$, 
and the action of $\varphi:\pi_1(G)\to \{\pm 1\}$ is the scalar $\varphi(*)\in \{\pm 1\}$ on the connected component corresponding to $*\in \pi_1(G)$. 
We will write the Coulomb branch algebra as
$\mathscr{A}_{G,\bfM}(\sqrt{\mathcal{L}})$ to emphasize the dependence of $\sqrt{\mathcal{L}}$.

When $\bfM$ is of cotangent type, say $\bfM=\bfN\oplus \bfN^*$, it gives a canonical choice of square root 
$$\sqrt{\mathcal{L}}_{(G,\bfN)}:= \pi^*\det,$$ where $\pi:G\to \GL(\bfN)$ and $\det$ is the determinant bundle over $\Gr_{GL(\bfN)}$. 
By \cite[Lemma 4.2.1]{BDFRT}, we can canonically identify
$$\mathscr{A}_{G,\bfM}(\sqrt{\mathcal{L}}_{(G,\bfN)})\cong \mathcal{A}_{\hbar=0}(G,\bfN).$$

Assume we have two choices of $\bfN$ as in \eqref{eq:twodecomp}. 
By \cite[Corollary 4.3.2]{BDFRT}, at least when $G$ is abelian, there exists an isomorphism $\varphi:\sqrt{\mathcal{L}}_{(G,\bfN_2)}\cong \sqrt{\mathcal{L}}_{(G,\bfN_1)}$ such that 
the isomorphism induced by $\sigma$ in \eqref{equ:sigma} coincides with
$$\mathcal{A}_{\hbar=0}(G,\bfN_2)
\stackrel{\sim}\longrightarrow 
\mathscr{A}_{G,\bfM}(\sqrt{\mathcal{L}}_{(G,\bfN_2)})
\stackrel{\varphi}\longrightarrow 
\mathscr{A}_{G,\bfM}
(\sqrt{\mathcal{L}}_{(G,\bfN_1)})
\stackrel{\sim}\longrightarrow 
\mathcal{A}_{\hbar=0}(G,\bfN_1).
$$

\subsection{Compatibility}
Now return to the separated cotangent setting. 
We will impose the technical assumption that 
\begin{equation}\label{eq:assumesemisimple}
\text{one of the groups $G_1,G_2$ is semisimple}. 
\end{equation}

\begin{theorem}\label{thm:compwithBDFRT}
Assuming \eqref{eq:assumesemisimple}, the Coulomb branch algebra of \cite{BDFRT} is isomorphic to 
$$
\mathcal{A}_{\hbar=0}(TG_1,\bfN_1)_{\operatorname{Frac}H_{T_2}^*(\pt)}
\bigcap
\mathcal{A}_{\hbar=0}
(TG_2,\bfN_2)_{\operatorname{Frac}H_{T_1}^*(\pt)}$$
with both components identified as a subalgebra of $\Diff_{\hbar=0}(T)$ via the $\hbar=0$ version of \eqref{eq:idenassubalgofDiff}. 
\end{theorem}

We need some preparations to prove this theorem.

\begin{lemma}\label{lem:factorLB}
As super factorization line bundles, 
we have 
$$\begin{aligned}
\mathcal{L}_{(G,\bfM)}& \cong \mathcal{L}_{(G_1,\bfM)}\boxtimes \mathcal{L}_{(G_2,\bfM)}
& \text{over $\Gr_G=\Gr_{G_1}\times \Gr_{G_2}$,}\\
\sqrt{\mathcal{L}}_{(TG_1,\bfN_1)}& \cong \sqrt{\mathcal{L}}_{(G_1,\bfN_1)}\boxtimes \sqrt{\mathcal{L}}_{(T_2,\bfN_1)}
& \text{over $\Gr_{TG_1}=\Gr_{G_1}\times \Gr_{T_2}$,}\\
\sqrt{\mathcal{L}}_{(TG_2,\bfN_2)}& \cong\sqrt{\mathcal{L}}_{(T_1,\bfN_2)}\boxtimes \sqrt{\mathcal{L}}_{(G_2,\bfN_2)}
& \text{over $\Gr_{TG_2}=\Gr_{T_1}\times \Gr_{G_2}$.}
\end{aligned}$$
Moreover, the isomorphism above is unique, 
if we assume it is identical when restricting to each factor. We will write ``$\cong$'' as ``$=$'' later. 
\end{lemma}
\begin{proof}
Recall the natural isomorphism used in 
\cite[Section 2]{BDFRT}, 
$$\operatorname{SPic}_{\rm fact}(\Gr_G)\cong 
\Bil(G).$$ Without loss of generality, assume $G_1$ is semisimple. 
Note that a semisimple group does not have any linear invariants, so it is easy to see 
$$\Bil(G_1\times G_2)
\cong \Bil(G_1)\oplus \Bil(G_2),
$$
The above isomorphism implies 
$\rho^*\det \cong \rho_1^*\det\boxtimes \rho_2^*\det$, for $\rho_i:G_i\to G\to \Sp(\bfM)$ i.e. 
$\mathcal{L}_{(G,\bfM)}
\cong \mathcal{L}_{(G_1,\bfM)}\boxtimes \mathcal{L}_{(G_2,\bfM)}$. 
The isomorphism is determined up to some group homomorphism $\pi_1(G)=\pi_1(G_1)\times \pi_1(G_2)\to \{\pm 1\}$. In particular, the isomorphism is determined by the restriction to $\Gr_{G_1}=\Gr_{G_1\times \{1\}}$ and $\Gr_{G_2}=\Gr_{\{1\}\times G_2}$. 

Similarly for the second isomorphism, as 
$$\Bil(G_1\times T_2)
\cong \Bil(G_1)\oplus \Bil(T_2).$$
Let us prove the last isomorphism. 
Note that the right-hand side 
$$\sqrt{\mathcal{L}}_{(T_1,\bfN_2)}\boxtimes \sqrt{\mathcal{L}}_{(G_2,\bfN_2)}$$
is already a square root of $\mathcal{L}_{(TG_2,\bfM)}$
by the first isomorphism. 
Since $\Bil(G)$ is torsion-free, it has to be isomorphic to the left-hand side. 
\end{proof}

By Lemma \ref{lem:factorLB}, we can take a square root 
$$\sqrt{\mathcal{L}}:=
\sqrt{\mathcal{L}}_{(G_1,\bfN_1)}\boxtimes 
\sqrt{\mathcal{L}}_{(G_2,\bfN_2)}.$$
For any subgroup $H$ of $G$, we write $-|_H$ the restriction from $\Gr_G$ to $\Gr_H$. 

Let us fix an isomorphism 
$$\varphi:\sqrt{\mathcal{L}}_{(T,\bfN_2)}
\stackrel{\sim}\longrightarrow
\sqrt{\mathcal{L}}_{(T,\bfN_1)}$$
such that it is compatible with the isomorphism $\sigma$ in \eqref{equ:sigma}. 

\begin{lemma}
Write $\varphi_i=\varphi|_{T_i}$. 
We have $\varphi=\varphi_1\boxtimes \varphi_2$. 
\end{lemma}

\begin{proof}
The isomorphisms $\varphi$ and 
$\varphi_1\boxtimes \varphi_2$ between square roots differ by the action of some group homomorphism $\pi_1(T)=\pi_1(T_1)\times \pi_1(T_2)\to \{\pm 1\}$. 
When restricting to 
$\Gr_{T_1}\times \{1\}$ and 
$\{1\}\times \Gr_{T_2}$, the two isomorphisms $\varphi$ and $\varphi_1\boxtimes \varphi_2$ coincide. 
In particular, the group homomorphism is trivial. 
\end{proof}

By Lemma \ref{lem:factorLB}, we get 
$$
\begin{aligned}
\operatorname{id}\boxtimes \varphi_2: \quad
\sqrt{\mathcal{L}}|_{TG_1}
& =\sqrt{\mathcal{L}}_{(G_1,\bfN_1)}\boxtimes 
\sqrt{\mathcal{L}}_{(G_2,\bfN_2)}|_{T_2}\\
& =\sqrt{\mathcal{L}}_{(G_1,\bfN_1)}\boxtimes 
\sqrt{\mathcal{L}}_{(T_2,\bfN_2)}\\
& \to 
\sqrt{\mathcal{L}}_{(G_1,\bfN_1)}\boxtimes 
\sqrt{\mathcal{L}}_{(T_2,\bfN_1)}
= \sqrt{\mathcal{L}}_{(TG_1,\bfN_1)},\\
\varphi_1^{-1}\boxtimes \operatorname{id}: \quad
\sqrt{\mathcal{L}}|_{TG_2}
& = 
\sqrt{\mathcal{L}}_{(G_1,\bfN_1)}|_{T_1}\boxtimes 
\sqrt{\mathcal{L}}_{(G_2,\bfN_2)}\\
& = 
\sqrt{\mathcal{L}}_{(T_1,\bfN_1)}\boxtimes 
\sqrt{\mathcal{L}}_{(G_2,\bfN_2)}\\
& \to 
\sqrt{\mathcal{L}}_{(T_1,\bfN_2)}\boxtimes 
\sqrt{\mathcal{L}}_{(G_2,\bfN_2)}
= \sqrt{\mathcal{L}}_{(TG_2,\bfN_2)}.
\end{aligned}
$$
\begin{prop}
Let $L$ be a Levi subgroup of $G$. We have natural maps
$$
\mathscr{A}_{(G,\bfM)}(\sqrt{\mathcal{L}})
\stackrel{\iota_*}\longleftarrow 
\mathscr{A}_{(L,\bfM)}(\sqrt{\mathcal{L}}|_L)^W
\stackrel{\subset}\longrightarrow 
\mathscr{A}_{(L,\bfM)}(\sqrt{\mathcal{L}}|_L).
$$
The first map is an isomorphism after localization. 
\end{prop}
\begin{proof}
This was done in \cite[Remark 4.1.5]{BDFRT} for torus case, but the argument can be generalized to any Levi subgroup. 
See also \cite[Proof of Lemma 4.4.1]{BDFRT}. 
\end{proof}

Let us denote by $\xymatrix@!=0.5pc{A\ar@{..>}[r]&B}$ if there is map defined after localization.
The above Proposition gives 
$$\xymatrix{\mathscr{A}_{(G,\bfM)}(\sqrt{\mathcal{L}})\ar@{..>}[r]^{(\iota_*)^{-1}} & 
\mathscr{A}_{(L,\bfM)}(\sqrt{\mathcal{L}}|_L).
}$$
We have the following commutative diagram
$$
\xymatrix@R=1.5pc@!C=4.5pc{
&&\mathcal{A}_{\hbar=0}(TG_2,\bfN_2)\ar@{..>}[dr]^{(\iota_{2,*})^{-1}}\\
&\mathscr{A}_{TG_2,\bfM}(\sqrt{\mathcal{L}}|_{TG_2})
\ar@{..>}[dr]
\ar[ur]^{\varphi_1^{-1}\boxtimes \operatorname{id}}&&
\mathcal{A}_{\hbar=0}(T,\bfN_2)\ar[dd]^{\varphi}
\ar[r]^-{z^*}&
\mathcal{A}_{\hbar=0}(T,0)
\ar[r]^-{\subset}&
\mathsf{Diff}_{\hbar=0}(T)\ar[dd]^{\sigma}
\\
\mathscr{A}_{G,\bfM}(\sqrt{\mathcal{L}})\ar@{..>}[dr]\ar@{..>}[ur]\ar@{..>}[rr]^{(\iota_*)^{-1}}& 
&
\mathscr{A}_{T,\bfM}(\sqrt{\mathcal{L}}|_T)
\ar[dr]_{\operatorname{id}\boxtimes \varphi_2}
\ar[ur]^{\varphi_1^{-1}\boxtimes \operatorname{id}}
&\\
&\mathscr{A}_{TG_1,\bfM}(\sqrt{\mathcal{L}}|_{TG_1})\ar@{..>}[ur]
\ar[dr]_{\operatorname{id}\boxtimes \varphi_2}&&
\mathcal{A}_{\hbar=0}(T,\bfN_1)
\ar[r]^-{z^*}&
\mathcal{A}_{\hbar=0}(T,0)
\ar[r]^-{\subset}&
\mathsf{Diff}_{\hbar=0}(T).\\
&&\mathcal{A}_{\hbar=0}(TG_1,\bfN_1)\ar@{..>}[ur]_{(\iota_{1,*})^{-1}}
}$$
Let us carry out the computation in the difference algebra $\Diff_{\hbar=0}(T)$, located in the lower-right corner of the diagram above.

Finally, we can prove Theorem \ref{thm:compwithBDFRT}.
\begin{proof}[Proof of Theorem \ref{thm:compwithBDFRT}]
Let us identify $\mathscr{A}_{G,\bfM}(\sqrt{\mathcal{L}})$ as a subalgebra of $\mathsf{Diff}_{\hbar=0}(T)$ via 
$$z^*\circ (\iota_{1,*})^{-1}\circ (\operatorname{id}\boxtimes \varphi_2) \circ (\iota_{*})^{-1}:
\mathscr{A}_{G,\bfM}(\sqrt{\mathcal{L}})
\stackrel{\subset}\longrightarrow \mathsf{Diff}_{\hbar=0}(T).$$
We leave it to the reader to check this map is $W$-invariant. 
Note that the map 
$$
\mathscr{A}_{G,\bfM}(\sqrt{\mathcal{L}})
\longleftarrow
\mathscr{A}_{TG_1,\bfM}(\sqrt{\mathcal{L}})^W
$$
is an isomorphism after tensoring $\operatorname{Frac}H_{T_2}^*(\pt)$. 
So we have 
$$\mathscr{A}_{G,\bfM}(\sqrt{\mathcal{L}})
\subset \mathcal{A}_{\hbar=0}(TG_1,\bfN_1)_{\operatorname{Frac}H_{T_2}^*(\pt)}.$$
Similarly, 
$$\mathscr{A}_{G,\bfM}(\sqrt{\mathcal{L}})
\subset \mathcal{A}_{\hbar=0}(TG_2,\bfN_2)_{\operatorname{Frac}H_{T_1}^*(\pt)}.$$
This establishes the inclusion ``$\subseteq$''.

The equality follows from the fact that $\mathscr{A}_{G,\bfM}(\sqrt{\mathcal{L}})$ is a normal ring \cite[Section 4.5]{BDFRT}. 
More precisely, assume 
$$X\in \mathcal{A}_{\hbar=0}(TG_1,\bfN_1)_{\operatorname{Frac}H_{T_2}^*(\pt)}
\bigcap
\mathcal{A}_{\hbar=0}
(TG_2,\bfN_2)_{\operatorname{Frac}H_{T_1}^*(\pt)}.$$
Then we can find nonzero 
$f_1\in H_{T_2}^*(\pt)$ and 
$f_2\in H_{T_1}^*(\pt)$ such that 
$$f_i X\in 
\mathcal{A}_{\hbar=0}(TG_i,\bfN_i).$$
Since $X$ itself is $W$-invariant, replacing $f_i$ by $\prod_{w\in W_{3-i}}w(f_i)\in H_{G_{3-i}}^*(\pt)$, we can further assume $f_iX$ is $W$-invariant. 
Since $\iota_*$ factors as 
$$
\mathscr{A}_{G,\bfM}(\sqrt{\mathcal{L}})
\longleftarrow
\mathscr{A}_{TG_i,\bfM}(\sqrt{\mathcal{L}}|_{TG_i})^{W}
\longleftarrow
\mathscr{A}_{T,\bfM}(\sqrt{\mathcal{L}}|_T)^{W},
$$
we have 
$$f_iX\in \mathcal{A}_{\hbar=0}(TG_i,\bfN_i)^W\subseteq \mathscr{A}_{G,\bfM}(\sqrt{\mathcal{L}}).$$
Over the variety 
(\cite[Lemma 4.4.1]{BDFRT}) 
$\operatorname{Spec}\mathscr{A}_{G,\bfM}(\sqrt{\mathcal{L}})$, the rational function
$X$ is defined over $\{f_1\neq 0\}$ and $\{f_2\neq 0\}$. Since $f_1,f_2$ are in different variables and $\mathscr{A}_{G,\bfM}(\sqrt{\mathcal{L}})$ is flat over $H_{G_1}^*(\pt)\otimes H_{G_2}^*(\pt)$ (see \cite[Remark 4.1.4]{BDFRT}), the union  $\{f_1\neq 0\}\cup\{f_2\neq 0\}$ has complement of codimension $\geq 2$. 
Since $\operatorname{Spec}\mathscr{A}_{G,\bfM}(\sqrt{\mathcal{L}})$ is normal (\cite[Lemma 4.5.1]{BDFRT}), $X$ is actually a regular function, i.e. $X\in \mathscr{A}_{G,\bfM}(\sqrt{\mathcal{L}})$.
\end{proof}

\section{Quantized Coulomb branch for the orthosymplectic quivers}\label{sec:orthosymp}

In this section, we compute some dressed monopole operators for the quantized Coulomb branch algebra associated to some orthosymplectic quivers.

\subsection{Orthosymplectic quivers}
We consider an orthosymplectic quiver $Q=(Q_0,Q_1,\epsilon)$. Here $Q$ is an unoriented simply-laced finite graph without edge loops, $Q_0$ is the set of vertices, $Q_1$ is the set of edges, and $\epsilon: Q_0\rightarrow \{\pm\}$ such that for any edge $i-j$, $\epsilon_i\neq \epsilon_j$.  
The following is an example of an orthosymplectic quiver.
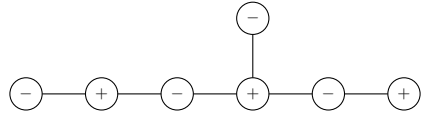
\[
\begin{tikzpicture}
    \draw[thin] (0,0) -- (5,0)
          (3,0) -- (3,1);
    \node[circle,draw,fill=white,scale=.5] at (0,0){$-$};
    \node[circle,draw,fill=white,scale=.5] at (2,0){$-$};
    \node[circle,draw,fill=white,scale=.5] at (4,0){$-$};
    \node[circle,draw,fill=white,scale=.5] at (1,0){$+$};
    \node[circle,draw,fill=white,scale=.5] at (3,0){$+$};
    \node[circle,draw,fill=white,scale=.5] at (5,0){$+$};
    \node[circle,draw,fill=white,scale=.5] at (3,1){$-$};
\end{tikzpicture}
\]
Let $\bbV,\bbW$ be two $Q_0$-graded vector spaces such that if $\epsilon_i=+/-$, then $\bbV_i$ (resp. $\bbW_i$) is equipped with a non-degenerate symmetric/antisymmetric (resp. antisymmetric/symmetric) bilinear form. 

Let us denote $K_{+}=\SO$ and $K_-=\Sp$.
We define 
$$\begin{aligned}
\text{gauge group }G & 
=\prod_{i\in Q_0}K_{\epsilon_i}(\bbV_i),\\
\text{matter representation }\mathbf{M} & 
=\bigoplus_{i- j}\bbV_i\otimes \bbV_j \oplus \bigoplus_{i\in Q_0}\bbV_i\otimes \bbW_i.\\
\end{aligned}$$

\begin{lemma}
    When $G = \Sp(V)$ the group $\Bil(G)$ is freely generated by the trace pairing $\operatorname{Tr}_V$. 
When $G = \SO(V)$ the trace pairing $\operatorname{Tr}_V\in 2\Bil(G)$.
\end{lemma}

\begin{proof}
Let $\check{\Lambda}\supset\check{Q}$ be the cocharacter lattice and the coroot lattice of $G$. 
By embedding $G\subseteq \GL(V)$, we can identify the cocharacter lattice of $G$ as a sublattice of $\mathbb{Z}^{\oplus n}$ if $\dim V=n$. 
The trace pairing can be identified with the standard dot product. 

When $G=\Sp(V)$ the cocharacter lattice coincides with the coroot lattice and is given by 
$$\check{Q}:=\{(a_1,\ldots,a_m,-a_1,\ldots,-a_m):a_i\in \mathbb{Z}\}$$
where $m=\dim V/2$. 
For any $\lambda,\mu\in \check{Q}$, it is easy to see
$$\operatorname{Tr}(\lambda,\mu)\in 2\mathbb{Z}\subseteq \mathbb{Z},\qquad 
\operatorname{Tr}(\lambda,\lambda)\in 2\mathbb{Z}.$$
Thus $\operatorname{Tr}_V\in \Bil(G)$. 
Since $G$ is simple, there is a unique $G$-invariant form up to scale. 
The Lemma follows from the fact that $\operatorname{Tr}(\lambda,\lambda)=2$ for $\lambda=(1,0,\ldots,-1,0,\ldots)$.

When $G=\SO(V)$ with $\dim V=n=2m+1$ odd, the cocharacter lattice is 
$$\check{\Lambda}:=\{(a_1,\ldots,a_m,0,-a_1,\ldots,-a_m):a_i\in \mathbb{Z}\}$$
while the coroot lattice $\check{Q}$ is the sublattice with $\sum a_i\equiv 0\bmod 2$. 
For any $\lambda,\mu\in \check{\Lambda}$ and $\nu\in \check{Q}$, 
$$\operatorname{Tr}(\lambda,\mu)\in 2\mathbb{Z},\qquad 
\operatorname{Tr}(\nu,\nu)\in 4\mathbb{Z}.$$
Thus we have $\operatorname{Tr}_V\in 2\Bil(G)$. 
The case when $\dim V=2m$ is even is similar. 
\end{proof}

\begin{prop}\label{prop:ano}
    The pair $(G,\mathbf{M})$ satisfies the anomaly cancellation condition \eqref{equ:anomaly} if and only if for any symplectic vertex $i$, we have
$$\ww_i+\sum_{j-i}\vv_j\equiv 0\mod 2.$$
Note that when $i$ is an orthogonal vertex, this is always true. 
\end{prop}
\begin{remark}
    This equivalent condition also appeared in Nakajima's various talks, for example, see \cite{Nak23}.
\end{remark}

\begin{proof}
Let us embed $G\subset \prod_{i\in Q_0}\GL(\bbV_i)$. 

Let 
$$\lambda = \sum_{i\in Q_0}\lambda_i,\text{ with } \lambda_i\in \mathfrak{gl}(\bbV_i).$$
be a coweight of $G$. 
Then 
$$\begin{aligned}
d\rho(\lambda) & = \sum_{i-j} 
(\lambda_i\otimes \operatorname{id}_{\bbV_j}
+\operatorname{id}_{\bbV_i}\otimes \lambda_j)+
\sum_{i\in Q_0}(\lambda_i\otimes \operatorname{id}_{\bbW_i})\\
& \in \bigoplus_{i-j} \operatorname{End}(\bbV_i\otimes \bbV_j)\oplus \bigoplus_{i\in Q_0}\operatorname{End}(\bbV_i\otimes \bbW_i)\subset \mathfrak{gl}(\mathbf{M}).
\end{aligned}$$
Let $\mu$ be another coweight. The matrix product is $$\begin{aligned} 
d\rho(\lambda)d\rho(\mu)
& = \sum_{i-j}
(\lambda_i\otimes \operatorname{id}_{\bbV_j}
+\operatorname{id}_{\bbV_i}\otimes \lambda_j)
(\mu_i\otimes \operatorname{id}_{\bbV_j}
+\operatorname{id}_{\bbV_i}\otimes \mu_j)\\
& +\sum_{i\in Q_0}(\lambda_i\otimes\operatorname{id}_{\bbW_i})
(\mu_i\otimes\operatorname{id}_{\bbW_i})
\\
& = \sum_{i-j}
(\lambda_i\mu_i\otimes \operatorname{id}_{\bbV_j}
+\operatorname{id}_{\bbV_i}\otimes \lambda_j\mu_j
+\lambda_i\otimes \mu_j
+\mu_i\otimes \lambda_j)\\
& +\sum_{i\in Q_0}(\lambda_i\mu_i\otimes\operatorname{id}_{\bbW_i})
\end{aligned}$$
Note that $\operatorname{tr}(X\otimes Y)=\operatorname{tr}(X)\cdot \operatorname{tr}(Y)$, and $\operatorname{Tr}(X)=0$ if $X$ is in the symplectic or orthogonal Lie algebra. 
The trace pairing is given by 
$$\begin{aligned}
\operatorname{Tr}(d\rho(\lambda),d\rho(\mu)) 
& = \sum_{i-j}
(\vv_j\operatorname{tr}(\lambda_i\mu_i)
+\vv_i\operatorname{tr}(\lambda_j\mu_j))
+ \sum_{i\in Q_0}
\ww_i\operatorname{tr}(\lambda_i\mu_i).\\
& = \sum_{i\in Q_0}
\left(\ww_i+\sum_{j-i}\vv_j\right)
\operatorname{tr}(\lambda_i\mu_i).
\end{aligned}$$
Since each component $\lambda_i$ of $\lambda$ is independent, the condition is equivalent to 
$$\forall i\in Q_0,\qquad 
\left(\ww_i+\sum_{j-i}\vv_j\right)
\operatorname{Tr}_{\bbV_i} \in 2 \Bil(K_{\epsilon_i}(V_i)).$$
When $i$ is orthogonal, by the above assumption, this is always true. 
When $i$ is symplectic, since $\operatorname{Tr}_{\bbV_i}$ is the generator of the group $\Bil(K_{\epsilon_i}(\bbV_i))$, so we get the condition.  
\end{proof}
From now on, we assume that $\{\bbV_i,\bbW_i\mid i\in Q_0\}$ satisfies the condition in Proposition \ref{prop:ano} throughout.

\subsection{Lagrangian}
It is useful to write 
$$\mathbf{M}=\bigoplus_{\epsilon_i=-}
\bbV_i\otimes \left(\bbW_i\oplus \bigoplus_{j-i}
\bbV_j\right)
\oplus \bigoplus_{\epsilon_i=+}
\bbW_i\otimes \bbV_i.
$$
Here we write symplectic spaces on the left of $\otimes$, and orthogonal spaces on the right of $\otimes$. 
For any vertex $i\in Q_0$, let us fix maximal isotropic subspaces $V_i\leq \bbV_i$ and $W_i\leq \bbW_i$. Let $v_i$ (resp. $w_i$) denote dimension of $V_i$ (resp. $W_i$), and let $\vv_i$ (resp. $\ww_i$) denote dimension of $\bbV_i$ (resp. $\bbW_i$). For each $n\in \mathbb{Z}_{\geq 0}$, let \[\bar{n}=\begin{cases}
    1, & \textit{if n is odd},\\
    0 , & \textit{if n is even}.
\end{cases}\] Then we have 
$$
\begin{aligned}
\bbV_i& \cong V_i\oplus V_i^* \text{ for $\epsilon_i=-$},&\qquad&&
\bbV_i& \cong V_i\oplus V_i^*\oplus \mathbb{C}^{\overline{\vv_i}} \text{ for $\epsilon_i=+$},\\ 
\bbW_i & \cong W_i\oplus W_i^* \text{ for $\epsilon_i=+$},&\qquad&&
\bbW_i& \cong W_i\oplus W_i^*\oplus \mathbb{C}^{\overline{\ww_i}} \text{ for $\epsilon_i=-$}.
\end{aligned}
$$
For each symplectic vertex $i$, we fix a space $\Upsilon_i$ such that 
$$
\mathbb{C}^{\overline{\ww_i}}\oplus \bigoplus_{j-i}\mathbb{C}^{\overline{\vv_j}}\cong 
\Upsilon_i \oplus \Upsilon_i^*.$$
The existence of $\Upsilon_i$ follows from the anomaly cancellation condition in Proposition \ref{prop:ano}. 

We will use a smaller flavor group $$F = \prod_{i}\GL(W_i).$$ 
Let us denote 
$$
G_+=\prod_{\epsilon_i=-}\GL(V_i)\times 
\prod_{\epsilon_i=+}\SO(\bbV_i),\qquad 
G_-=\prod_{\epsilon_i=-}\Sp(\bbV_i)\times \prod_{\epsilon_i=+}\GL(V_i).$$
We consider 
$$
\begin{aligned}
\bfN_+ & = \bigoplus_{\epsilon_i=-}
V_i\otimes \left(\bbW_i\oplus \bigoplus_{j-i}\bbV_j\right)
\oplus 
\bigoplus_{\epsilon_i=+}
W_i\otimes \bbV_i,\\
\bfN_- & = \bigoplus_{\epsilon_i=-}
\bbV_i\otimes \left(W_i\oplus \bigoplus_{j-i}V_j\oplus \Upsilon_i\right)
\oplus 
\bigoplus_{\epsilon_i=+}
W_i\otimes \bbV_i.\\
\end{aligned}$$
We can naturally identify $\bfN_{\pm}$ as a $G_{\pm}\times F$-Lagrangian of $\mathbf{M}$. 
Let $G_0=G_+\cap G_-=\prod_{i\in Q_0}\GL(V_i)$. 
Then $G_0$ acts on 
$$\bfN_0:=\bfN_+\cap \bfN_-
= \bigoplus_{\epsilon_i=-}
V_i\otimes \left(W_i\oplus \bigoplus_{j-i}V_j\oplus \Upsilon_i\right)
\oplus 
\bigoplus_{\epsilon_i=+}
W_i\otimes \bbV_i.$$
Moreover, 
$$
\begin{aligned}
\bfN_+ &
=\bfN_0\oplus \bigoplus_{\epsilon_i=-}
V_i\otimes \left(W_i\oplus \bigoplus_{j-i}V_j\oplus \Upsilon_i\right)^*,\\
\bfN_- &
=\bfN_0\oplus \bigoplus_{\epsilon_i=-}
V_i^*\otimes \left(W_i\oplus \bigoplus_{j-i}V_j\oplus \Upsilon_i\right),
\end{aligned}$$
respectively. 
We let \[\underline{\bfN}:=\bigoplus_{\epsilon_i=-}
V_i\otimes \left(W_i\oplus \bigoplus_{j-i}V_j\oplus \Upsilon_i\right)^*.\]

For each $i\in Q_0$, we choose a basis $\{e_{i,r}, f_{i,r}\mid 1\leq r\leq v_i\}$ for $\bbV_i$ if $\bbV_i$ has even dimension,  and $\{e_{i,r}, f_{i,r}, u_i\mid 1\leq r\leq v_i\}$ if $\bbV_i$ has odd dimension. Here $\{e_{i,r}\mid 1\leq r\leq v_i\}$ is a basis for $V_i$ and $\{f_{i,r}\mid 1\leq r\leq v_i\}$ is a basis for $V_i^*$ such that $(e_{i,r},f_{i,s})=\delta_{r,s}$, the vector $u_i$ satisfies $(u_i,u_i)=1$. We do the same for the vector space $\bbW_i$. This determines a maximal torus $T_{\epsilon_i}(\bbV_i)=\{\diag(z_{i,1},\cdots, z_{i,v_i},z_{i,1}^{-1},\cdots, z_{i,v_i}^{-1})\}\subset K_{\epsilon_i}(\bbV_i)$ if $\bbV_i$ has even dimension, and $T_+(\bbV_i)=\{\diag(z_{i,1},\cdots, z_{i,v_i},1,z_{i,1}^{-1},\cdots, z_{i,v_i}^{-1})\}\subset K_+(\bbV_i)$ if $\bbV_i$ has odd dimension. We let $x_{i,r}$ denote the $r$-th standard character of $\Lie(T_{\epsilon_i}(\bbV_i))$. 
Hence, $x_{i,r}+x_{i,r+v_i+\overline{\vv_i}}=0$ for $1\leq r\leq v_i$, and $x_{i,1+v_i}=0$ if $\vv_i$ is odd. Similarly, we let $w_{i,r}$ denote the standard character of the Lie algebra of the maximal torus of $K_{-\epsilon_i}(\bbW_i)$. We use $\epsilon_{i,r}$ to denote the $r$-th cocharacter of $T_{\epsilon_i}(\bbV_i)$. Notice that $\epsilon_{i,1}$ is a minuscule coweight for $K_+(\bbV_i)$.

\subsection{Monopole operators}

Now let us consider the quantized Coulomb branch algebra $\mathscr{A}^F_\hbar(G,\bfM)$ for the pair $(G,\bfM)$ with flavor symmetry group $F$. Let 
\begin{equation}\label{equ:G1G2}
    G_1:=\prod_{\epsilon(i)=-}\Sp(\bbV_i), \quad G_2:=\prod_{\epsilon(i)=+}\SO(\bbV_i),
\end{equation}
and let 
\[\bfN_1:=\bfN_-, \quad\bfN_2=\bfN_+.\]
Then $(G,M)$ is of separated cotangent type and satisfies the condition \eqref{eq:assumesemisimple}.

In this section, we compute the monopole operator associated to the first fundamental coweight $\epsilon_{i,1}$ for each vertex. We use $d_{i,r}$ to denote $d_{\epsilon_{i,r}}$ for $1\leq r\leq \vv_i$. Then $d_{i,v_i+r+\overline{\vv_i} }=d_{i,r}^{-1}$ for $1\leq r\leq v_i$, and $d_{i,1+v_i}=1$ if $\vv_i$ is odd.

Recall that a lattice $L$ in $\bbV_i(\calK)$ is an $\calO$-submodule in $\bbV_i(\calK)$ such that $L\otimes_{\calO}\calK=\bbV_i(\calK)$. For a lattice $L$, we define its dual $L^\vee:=\{x\in \bbV_i(\calK)\mid (x,L)\subseteq \calO\}$, where the pairing $(-,-)$ is induced from the one on $\bbV_i$. Then the affine Grassmannian $\Gr_{K_{\epsilon_i}(\bbV_i)}$ is the set of self-dual lattices in $\bbV_i(\calK)$, and the Schubert cell
\begin{align*}
    \Gr_{K_{\epsilon_i}(\bbV_i),\epsilon_{i,1}}=&\{\textit{self-dual lattices }  L\subset \bbV_i(\calK)\mid\\
    &z\bbV_i(\calO)\subset L\subset z^{-1}\bbV_i(\calO), \dim_{\mathbb{C}} \frac{\bbV_i(\calO)}{\bbV_i(\calO)\cap L}=1\}.
\end{align*}
There is a tautological line bundle $\mathcal{L}_i$ on $\Gr_{K_{\epsilon_i}(\bbV_i),\epsilon_{i,1}}$, whose fiber at a lattice $L$ is $\frac{\bbV_i(\calO)}{\bbV_i(\calO)\cap L}$. Under the isomorphism $H_G^*(\Gr_{K_{\epsilon_i}(\bbV_i),\epsilon_{i,1}})\simeq H_G^*(G/P_{\epsilon_{i,1}})\simeq \mathbb{C}[\mathfrak{t}]^{W_{\epsilon_{i,1}}}$, the first Chern class $c_1(\mathcal{L}_i)$ is given by $x_{i,1}$. The torus fixed points in $\Gr_{K_{\epsilon_i}(\bbV_i),\epsilon_{i,1}}$ are $\{z^{\pm\epsilon_{i,r}}K_{\epsilon_i}(\bbV_i)(\calO)/K_{\epsilon_i}(\bbV_i)(\calO)\mid 1\leq r\leq v_i\}$. To simplify notation, we will just use $z^{\pm\epsilon_{i,r}}$ to denote the corresponding fixed points, if there is no confusion.

Let us introduce the following functions:
\[V_i(z):=\prod_{s=1}^{v_i}(z-x_{i,s}), \quad W_i(z):=\prod_{s=1}^{w_i}(z-w_{i,s}),\]
\[\bbV_i(z):=z^{\overline{\vv_i}}\prod_{s=1}^{v_i}(z^2-x^2_{i,s}), \quad \bbW_i(z):=z^{\overline{\ww_i}}\prod_{s=1}^{w_i}(z^2-w^2_{i,s}),\]
and 
\[\bbV_{i,r}(z):=\frac{\bbV_i(z)}{z-x_{i,r}}, \quad \bbW_{i,r}(z):=\frac{\bbW_i(z)}{z-w_{i,r}}.\]

\subsubsection{Orthogonal vertex}
Recall $G_+=\prod_{\epsilon_i=-}\GL(V_i)\times 
\prod_{\epsilon_i=+}\SO(\bbV_i)$. Let us choose a vertex $i\in Q_0$ such that $\epsilon_i=+$. Then $\epsilon_{i,1}$ is a minuscule coweight for $K_{\epsilon_i}(\bbV_i)=\SO(\bbV_i)$. We have the following formulae. 

\begin{prop}\label{prop:orth}
    Fix a polynomial $f(x)\in \mathbb{C}[x]$.
    \begin{enumerate}
        \item If $\dim \bbV_i=2v_i$, then \begin{align*}
        \sigma\circ z^*&\circ (\iota_*)^{-1}\bigg(f(c_1(\mathcal{L}_i))\cdot [\calR_{G_+,\bfN_+,\epsilon_{i,1}}]\bigg)=\sum_{r=1}^{v_i} f(x_{i,r})\frac{(-1)^{w_i}W_i(x_{i,r}+\frac{\hbar}{2})}{\bbV_{i,r}(x_{i,r})/2x_{i,r}}d_{i,r}\\&+\sum_{r=1}^{v_i} f(-x_{i,r})(-1)^{w_i+\sum_{k-i}v_k}\frac{W_i(-x_{i,r}+\frac{\hbar}{2})\prod_{k-i}\bbV_k(x_{i,r}-\frac{\hbar}{2})}{\bbV_{i,r}(x_{i,r})/2x_{i,r}}d_{i,r}^{-1}
        \in \mathscr{A}^F_\hbar(G,\bfM). 
    \end{align*}
    \item 
    If $\dim \bbV_i=2v_i+1$, then \begin{align*}
        \sigma\circ z^*&\circ (\iota_*)^{-1}\bigg(f(c_1(\mathcal{L}_i))\cdot [\calR_{G_+,\bfN_+,\epsilon_{i,1}}]\bigg)
        =\sum_{r=1}^{v_i} f(x_{i,r})\frac{(-1)^{w_i}W_i(x_{i,r}+\frac{\hbar}{2})}{\bbV_{i,r}(x_{i,r})/2x_{i,r}}d_{i,r}\\
        &-\sum_{r=1}^{v_i} f(-x_{i,r})(-1)^{w_i+\sum_{k-i}v_k}\frac{W_i(-x_{i,r}+\frac{\hbar}{2})\prod_{k-i}\bbV_k(x_{i,r}-\frac{\hbar}{2})}{\bbV_{i,r}(x_{i,r})/2x_{i,r}}d_{i,r}^{-1}\in \mathscr{A}^F_\hbar(G,\bfM). 
    \end{align*}
    \end{enumerate}   
\end{prop}
\begin{proof}
    Let us first assume that $\dim \bbV_i=2v_i$ is even. 
    Since $\bfN_+=\bfN_0\oplus \underline{\bfN}$ and $\bfN_-=\bfN_0\oplus \underline{\bfN}^*$, by \eqref{eq:minu} and \eqref{equ:sigma}, we get
    \begin{align}\label{equ:szl}
        &\sigma\circ z^*\circ (\iota_*)^{-1}\bigg(f(c_1(\mathcal{L}_i))\cdot [\calR_{G_+,\bfN_+,\epsilon_{i,1}}]\bigg)\\
        =&\sum_{r=1}^{v_i} f(x_{i,r})\frac{(-1)^{\sum_{\chi}\max (\chi(\epsilon_{i,r}),0)\dim\underline{\bfN}(\chi)}}{\Eu(T_{\epsilon_{i,r}}\Gr_{\epsilon_{i,1}})}\Eu\bigg(z^{\epsilon_{i,r}}\bfN_-(\calO)/(z^{\epsilon_{i,r}}\bfN_-(\calO)\cap \bfN_-(\calO))\bigg)d_{i,r}\notag\\
        +&\sum_{r=1}^{v_i} f(-x_{i,r})\frac{(-1)^{\sum_{\chi}\max (\chi(-\epsilon_{i,r}),0)\dim\underline{\bfN}(\chi)}}{\Eu(T_{-\epsilon_{i,r}}\Gr_{\epsilon_{i,1}})}\Eu\bigg(z^{-\epsilon_{i,r}}\bfN_-(\calO)/(z^{-\epsilon_{i,r}}\bfN_-(\calO)\cap \bfN_-(\calO))\bigg)d_{i,r}^{-1}.\notag
    \end{align}
    Here $T_{\pm\epsilon_{i,r}}\Gr_{\epsilon_{i,1}}$ denotes the tangent space of $\Gr_{\SO(\bbV_i),\epsilon_{i,1}}$ at the torus fixed points $\pm\epsilon_{i,r}$. By formula \eqref{equ:wts},
    \[\Eu(T_{\pm\epsilon_{i,r}}\Gr_{\epsilon_{i,1}})=\prod_{1\leq s\neq r\leq v_i}(x_{i,r}^2-x_{i,s}^2)=\frac{\bbV_{i,r}(x_{i,r})}{2x_{i,r}}.\]
    Since \[\underline{\bfN}:=\bigoplus_{\epsilon_k=-}V_k\otimes \left(W_k\oplus \bigoplus_{j-k}V_j\oplus \Upsilon_k\right)^*,\] and for each symplectic vertex $k$,  $$\mathbb{C}^{\overline{\ww_k}}\oplus \bigoplus_{j-k}\mathbb{C}^{\overline{\vv_j}}\cong \Upsilon_k \oplus \Upsilon_k^*,$$
    we get
    \[\sum_{\chi}\max (\chi(\epsilon_{i,r}),0)\dim\underline{\bfN}(\chi)=0,\]
    and 
    \[\sum_{\chi}\max (\chi(-\epsilon_{i,r}),0)\dim\underline{\bfN}(\chi)=\sum_{k-i}v_k.\]
    Recall \[\bfN_- = \bigoplus_{\epsilon_k=-}\bbV_k\otimes \left(W_k\oplus \bigoplus_{j-k}V_j\oplus \Upsilon_k\right)\oplus \bigoplus_{\epsilon_k=+}W_k\otimes \bbV_k.\]
    By \eqref{equ:Euquo}, only the term $W_i\otimes V_i^*\subset W_i\otimes \bbV_i\subset \bfN_-$ has contribution to \[\Eu\bigg(z^{\epsilon_{i,r}}\bfN_-(\calO)/(z^{\epsilon_{i,r}}\bfN_-(\calO)\cap \bfN_-(\calO))\bigg),\]
    and its contribution is 
    \[\prod_{1\leq a\leq w_i}(w_{i,a}-x_{i,r}-\frac{\hbar}{2})=(-1)^{w_i}W_i(x_{i,r}+\frac{\hbar}{2}).\]
    Similarly, only the term $\bigoplus_{k-i}\bbV_k\otimes V_i\oplus W_i\otimes V_i\subset \bfN_-$ has contribution to \[\Eu\bigg(z^{-\epsilon_{i,r}}\bfN_-(\calO)/(z^{-\epsilon_{i,r}}\bfN_-(\calO)\cap \bfN_-(\calO))\bigg),\]
    and its contribution is
    \begin{align*}
        &\prod_{k-i}\prod_{1\leq a\leq v_k}(x_{k,a}+x_{i,r}-\frac{\hbar}{2})(-x_{k,a}+x_{i,r}-\frac{\hbar}{2})\prod_{1\leq b\leq w_i}(w_{i,b}+x_{i,r}-\frac{\hbar}{2})\\
        =&(-1)^{w_i}W_i(-x_{i,r}+\frac{\hbar}{2})\prod_{k-i}\bbV_k(x_{i,r}-\frac{\hbar}{2}).
    \end{align*}
    Substituting these into \eqref{equ:szl}, we get the desired formula.

    For the case $\dim \bbV_i=2v_i+1$ is odd, we only need to notice that the only difference is 
    \[\Eu(T_{\pm\epsilon_{i,r}}\Gr_{\epsilon_{i,1}})=\pm x_{i,r}\prod_{1\leq s\neq r\leq v_i}(x_{i,r}^2-x_{i,s}^2)=\pm\frac{\bbV_{i,r}(x_{i,r})}{2x_{i,r}}.\]

    Finally, let us show that these operators belong to $\mathscr{A}^F_\hbar(G,\bfM)$. Recall the definition of $G_i$ in $\eqref{equ:G1G2}$. By definition, these operators are invariant under the twisted action of the Weyl group of $G_2$. It is obvious that these operators are invariant under the usual Weyl group action of $G_1$. Hence, by Theorem \ref{thm:noncotsub}, we get that they do belong to $\mathscr{A}^F_\hbar(G,\bfM)$.
\end{proof}

\subsubsection{Symplectic vertex}
Recall $G_-=\prod_{\epsilon_i=-}\Sp(\bbV_i)\times \prod_{\epsilon_i=+}\GL(V_i),$
and
\[
\bfN_-  = \bigoplus_{\epsilon_i=-}
\bbV_i\otimes \left(W_i\oplus \bigoplus_{j-i}V_j\oplus \Upsilon_i\right)
\oplus 
\bigoplus_{\epsilon_i=+}
W_i\otimes \bbV_i.\]
Let us choose a vertex $i$ such that $\epsilon(i)=-$.
We will compute the monopole operator associated to the first fundamental coweight $\epsilon_{i,1}$ of $\Sp(\bbV_i)$, which is a quasi-minuscule coweight. By Theorem \ref{thm:quasi}, we get the following explicit formula, where we use the notations from Section \ref{sec:quasi}.
\begin{prop}\label{prop:symp}
    For any polynomial $f(x)\in \mathbb{C}[x]$,
    \begin{align*}
        &z^*\circ (\iota_*)^{-1}\big(\tilde{p}_*(f(c_1(\mathcal{L}_i)+\frac{\hbar}{2})\cdot [\widetilde{\calR}_{G_-,\bfN_-,\leq \epsilon_{i,1}}])\big)\\
        =&(-1)^{\frac{\ww_i+\sum_{j-i}\vv_j}{2}}\sum_{r=1}^{\vv_i}f(x_{i,r}+\frac{\hbar}{2})\frac{(x_{i,r}+\frac{\hbar}{2})^{\dim \Upsilon_i}W_i(x_{i,r}+\frac{\hbar}{2})\prod_{j-i}V_j(x_{i,r}+\frac{\hbar}{2})}{(2x_{i,r}+\hbar)\bbV_{i,r}(x_{i,r})}(d_{i,r}-1)\\
        &+\cdots\in \mathcal{A}^F_\hbar(G_-,\bfN_-),
    \end{align*}
    where $\cdots$ denotes some polynomial in $H^*_{G_-\times\mathbb{C}^\times_\hbar}(\pt)$.
\end{prop}
\begin{proof}
    By Theorem \ref{thm:quasi}, we have
    \begin{align}\label{equ:ziota}
        &z^*\circ (\iota_*)^{-1}\big(\tilde{p}_*(f(c_1(\mathcal{L}_i)+\frac{\hbar}{2})\cdot [\widetilde{\calR}_{G_-,\bfN_-,\leq \epsilon_{i,1}}])\big)\\
        =&\sum_{w\in W^{\epsilon_{i,1}}}w\left(f(x_{i,1}+\frac{\hbar}{2})\cdot \frac{\Eu\big(z^{\epsilon_{i,1}}\bfN_-(\calO)/(z^{\epsilon_{i,1}}\bfN_-(\calO)\cap\bfN_-(\calO))\big)}{\Eu(T_{\epsilon_{i,1}}\Gr_{\epsilon_{i,1}})}(d_{i,1}-1)\right)
        +\cdots
        \notag\\
        =&\sum_{r=1}^{v_i}f(x_{i,r}+\frac{\hbar}{2})\cdot \frac{\Eu\big(z^{\epsilon_{i,r}}\bfN_-(\calO)/(z^{\epsilon_{i,r}}\bfN_-(\calO)\cap\bfN_-(\calO))\big)}{\Eu(T_{\epsilon_{i,r}}\Gr_{\epsilon_{i,1}})}(d_{i,r}-1)\notag\\
        &+\sum_{r=1}^{v_i}f(-x_{i,r}+\frac{\hbar}{2})\cdot \frac{\Eu\big(z^{-\epsilon_{i,r}}\bfN_-(\calO)/(z^{-\epsilon_{i,r}}\bfN_-(\calO)\cap\bfN_-(\calO))\big)}{\Eu(T_{-\epsilon_{i,r}}\Gr_{\epsilon_{i,1}})}(d_{i,r}^{-1}-1)+\cdots.\notag
    \end{align}
    By formula \eqref{equ:wts},
    \[\Eu(T_{\epsilon_{i,r}}\Gr_{\epsilon_{i,1}})=2x_{i,r}(2x_{i,r}+\hbar)\prod_{1\leq s\neq r\leq v_i}(x_{i,r}^2-x_{i,s}^2)=(2x_{i,r}+\hbar)\bbV_{i,r}(x_{i,r}),\]
    and \[\Eu(T_{-\epsilon_{i,r}}\Gr_{\epsilon_{i,1}})=-2x_{i,r}(-2x_{i,r}+\hbar)\prod_{1\leq s\neq r\leq v_i}(x_{i,r}^2-x_{i,s}^2)=(2x_{i,r}-\hbar)\bbV_{i,r}(x_{i,r}).\]
    Recall \[\bfN_- = \bigoplus_{\epsilon_k=-}\bbV_k\otimes \left(W_k\oplus \bigoplus_{j-k}V_j\oplus \Upsilon_k\right)\oplus \bigoplus_{\epsilon_k=+}W_k\otimes \bbV_k.\]
    By \eqref{equ:Euquo}, only the term $V_i^*\otimes (W_i\oplus\bigoplus_{j-i}V_j\oplus \Upsilon_i)\subset \bfN_-$ has contribution to \[\Eu\bigg(z^{\epsilon_{i,r}}\bfN_-(\calO)/(z^{\epsilon_{i,r}}\bfN_-(\calO)\cap \bfN_-(\calO))\bigg),\]
    and its contribution is 
    \begin{align*}
        &(-x_{i,r}-\frac{\hbar}{2})^{\dim \Upsilon_i}\prod_{1\leq a\leq w_i}(w_{i,a}-x_{i,r}-\frac{\hbar}{2})\prod_{j-i}\prod_{1\leq b\leq v_j}(x_{j,b}-x_{i,r}-\frac{\hbar}{2})\\
        =&(-x_{i,r}-\frac{\hbar}{2})^{\dim \Upsilon_i}(-1)^{w_i}W_i(x_{i,r}+\frac{\hbar}{2})\prod_{j-i}(-1)^{v_j}V_j(x_{i,r}+\frac{\hbar}{2})\\
        =&(-1)^{\frac{\ww_i+\sum_{j-i}\vv_j}{2}}(x_{i,r}+\frac{\hbar}{2})^{\dim \Upsilon_i}W_i(x_{i,r}+\frac{\hbar}{2})\prod_{j-i}V_j(x_{i,r}+\frac{\hbar}{2}),
    \end{align*}
    where the last equality follows from the fact that $\dim \Upsilon_i+w_i+\sum_{j-i}v_j=\frac{\ww_i+\sum_{j-i}\vv_j}{2}$.
    Similarly, only the term $V_i\otimes (W_i\oplus\bigoplus_{j-i}V_j\oplus \Upsilon_i)\subset \bfN_-$ has contribution to \[\Eu\bigg(z^{-\epsilon_{i,r}}\bfN_-(\calO)/(z^{-\epsilon_{i,r}}\bfN_-(\calO)\cap \bfN_-(\calO))\bigg),\]
    and its contribution is
    \begin{align*}
        &(x_{i,r}-\frac{\hbar}{2})^{\dim \Upsilon_i}\prod_{1\leq a\leq w_i}(w_{i,a}+x_{i,r}-\frac{\hbar}{2})\prod_{j-i}\prod_{1\leq b\leq v_j}(x_{j,b}+x_{i,r}-\frac{\hbar}{2})\\
        =&(x_{i,r}-\frac{\hbar}{2})^{\dim \Upsilon_i}(-1)^{w_i}W_i(-x_{i,r}+\frac{\hbar}{2})\prod_{j-i}(-1)^{v_j}V_j(-x_{i,r}+\frac{\hbar}{2})\\
        =&(-1)^{\frac{\ww_i+\sum_{j-i}\vv_j}{2}}(-x_{i,r}+\frac{\hbar}{2})^{\dim \Upsilon_i}W_i(-x_{i,r}+\frac{\hbar}{2})\prod_{j-i}V_j(-x_{i,r}+\frac{\hbar}{2}).
    \end{align*}
    Substituting these into \eqref{equ:ziota}, we get the desired formula.
\end{proof}

\subsubsection{Operators}
To summarize, we get the following explicit operators in the quantized Coulomb branch. 
\begin{prop}\label{prop:Bi}
    \begin{enumerate}
        \item For an orthogonal vertex $i$,
        \begin{align*}
        &B_i(u):=-\sum_{r=1}^{v_i} \frac{1}{u+\frac{\hbar}{2}+x_{i,r}}\frac{W_i(x_{i,r}+\frac{\hbar}{2})}{\bbV_{i,r}(x_{i,r})/2x_{i,r}}d_{i,r}\\
        &-\sum_{r=1}^{v_i} \frac{(-1)^{\vv_i+\sum_{k-i}v_k}}{u+\frac{\hbar}{2}-x_{i,r}}\frac{W_i(-x_{i,r}+\frac{\hbar}{2})\prod_{k-i}\bbV_k(x_{i,r}-\frac{\hbar}{2})}{\bbV_{i,r}(x_{i,r})/2x_{i,r}}d_{i,r}^{-1}\in \mathscr{A}^F_\hbar(G,\bfM)[[u^{-1}]]. 
    \end{align*}
    \item For a symplectic vertex $i$,
    \begin{align*}
        B_i(u):
        =&-\sum_{r=1}^{\vv_i}\frac{1}{u+\frac{\hbar}{2}+x_{i,r}}\frac{(x_{i,r}+\frac{\hbar}{2})^{\dim \Upsilon_i}W_i(x_{i,r}+\frac{\hbar}{2})\prod_{j-i}V_j(x_{i,r}+\frac{\hbar}{2})}{(x_{i,r}+\frac{\hbar}{2})\bbV_{i,r}(x_{i,r})}d_{i,r}\\
        &-\delta_{\dim \Upsilon_i,0}\frac{W_i(0)\prod_{j-i}V_j(0)}{u\cdot \bbV_i(\frac{\hbar}{2})}\\
        =&-\sum_{r=1}^{v_i}\frac{1}{u+\frac{\hbar}{2}+x_{i,r}}\frac{(x_{i,r}+\frac{\hbar}{2})^{\dim \Upsilon_i}W_i(x_{i,r}+\frac{\hbar}{2})\prod_{j-i}V_j(x_{i,r}+\frac{\hbar}{2})}{(x_{i,r}+\frac{\hbar}{2})\bbV_{i,r}(x_{i,r})}d_{i,r}\\
        &-\sum_{r=1}^{v_i}\frac{1}{u+\frac{\hbar}{2}-x_{i,r}}\frac{(-x_{i,r}+\frac{\hbar}{2})^{\dim \Upsilon_i}W_i(-x_{i,r}+\frac{\hbar}{2})\prod_{j-i}V_j(-x_{i,r}+\frac{\hbar}{2})}{(x_{i,r}-\frac{\hbar}{2})\bbV_{i,r}(x_{i,r})}d_{i,r}^{-1}\\
        &-\delta_{\dim \Upsilon_i,0}\frac{W_i(0)\prod_{j-i}V_j(0)}{u\cdot \bbV_i(\frac{\hbar}{2})}\in 
        \mathscr{A}^F_\hbar(G,\bfM)[[u^{-1}]].
    \end{align*}
    \end{enumerate}
\end{prop}
\begin{proof}
    The first one follows from Proposition \ref{prop:orth}(1) by taking $f(x)=-(-x-\frac{\hbar}{2})^s$, multiplying by $u^{-s-1}$, and summing over all $s\geq 0$. 
    
    Now let us prove the second one.
    Notice that if $\dim \Upsilon_i\geq 1$,
    \[\sum_{r=1}^{\vv_i}f(x_{i,r}+\frac{\hbar}{2})\frac{(x_{i,r}+\frac{\hbar}{2})^{\dim \Upsilon_i}W_i(x_{i,r}+\frac{\hbar}{2})\prod_{j-i}V_j(x_{i,r}+\frac{\hbar}{2})}{(x_{i,r}+\frac{\hbar}{2})\bbV_{i,r}(x_{i,r})}\]
    belongs to $H^*_{G_-\times\mathbb{C}^\times_\hbar}(\pt)$ by the same argument at the end of the proof for Theorem \ref{thm:quasi}. If $\dim \Upsilon_i=0$,
    \begin{align*}
    &\sum_{r=1}^{\vv_i}f(x_{i,r}+\frac{\hbar}{2})\frac{W_i(x_{i,r}+\frac{\hbar}{2})\prod_{j-i}V_j(x_{i,r}+\frac{\hbar}{2})}{(x_{i,r}+\frac{\hbar}{2})\bbV_{i,r}(x_{i,r})}\\
    =&f(0)W_i(0)\prod_{j-i}V_j(0)\sum_{r=1}^{\vv_i}\frac{1}{(x_{i,r}+\frac{\hbar}{2})\bbV_{i,r}(x_{i,r})}+\cdots\\
    =&f(0)W_i(0)\prod_{j-i}V_j(0)\frac{1}{\prod_{r=1}^{\vv_i}(x_{i,r}+\frac{\hbar}{2})}\sum_{r=1}^{\vv_i}\frac{\prod_{1\leq s\neq r\leq \vv_i}(x_{i,s}+\frac{\hbar}{2})}{\bbV_{i,r}(x_{i,r})}+\cdots\\
    =&-\frac{f(0)W_i(0)\prod_{j-i}V_j(0)}{\bbV_i(\frac{\hbar}{2})}+\cdots.
    \end{align*}
    Here $\cdots$ denotes some term in $H^*_{G_-\times\mathbb{C}^\times_\hbar}(\pt)$, and the last equality follows from
    \begin{align*}
    \sum_{r=1}^{\vv_i}\frac{\prod_{1\leq s\neq r\leq \vv_i}(x_{i,s}+\frac{\hbar}{2})}{\bbV_{i,r}(x_{i,r})}
    =-1.
    \end{align*}
    Therefore, by the same argument as the first one, we see that 
    \[B_i(u)\in \mathcal{A}_\hbar^F(G_-,\bfN_-)[[u^{-1}]].\]
    
    It is obvious that $B_i(u)$ is invariant under the usual action of the Weyl group of $G_1$ in \eqref{equ:G1G2}. Let us show that it is also invariant under the twisted action of the Weyl group of $G_2$. Let $w$ be an element in the Weyl group of $G_2$. Then it fixes each $\epsilon_{i,r}$ and $x_{i,r}$. By \eqref{equ:twistedw2} and the proof in Proposition \ref{prop:symp}, $w$ fixes the term 
    \[\frac{(x_{i,r}+\frac{\hbar}{2})^{\dim \Upsilon_i}W_i(x_{i,r}+\frac{\hbar}{2})\prod_{j-i}V_j(x_{i,r}+\frac{\hbar}{2})}{(x_{i,r}+\frac{\hbar}{2})\bbV_{i,r}(x_{i,r})}d_{i,r}\]
    in the summand of $B_i(u)$. Now let us look at the constant term of $B_i(u)$:
    \[\delta_{\dim \Upsilon_i,0}\frac{W_i(0)\prod_{j-i}V_j(0)}{u\cdot \bbV_i(\frac{\hbar}{2})}.\]
    We only need to consider the case $\dim\Upsilon_i=0$, which implies that $\overline{\vv_j}=0$ for any $j$ adjacent to $i$. Hence, $V_j(0)=(-1)^{v_j}\prod_{s=1}^{v_j}x_{j,s}$ is invariant under the Weyl group action of $\SO(\bbV_j)$. Therefore, we get that $B_i(u)\in \mathscr{A}^F_\hbar(G,\bfM)[[u^{-1}]]$ by Theorem \ref{thm:noncotsub}.
\end{proof}
We expand $B_i(u)$ as a Laurent series in $u^{-1}$, and let $B_{i,r}$ denote the coefficient of $u^{-r-1}$ in $B_i(u)$.
\begin{definition}
    \begin{enumerate}
        \item For an orthogonal vertex, define 
        \[H_i(u):=(-1)^{\vv_i+\sum_{k-i}v_k}\frac{(4u^2-\hbar^2)W_i(u)W_i(-u)\prod_{k-i}\bbV_k(u)}{\bbV_i(u-\frac{\hbar}{2})\bbV_i(u+\frac{\hbar}{2})}\in \mathscr{A}^F_\hbar(G,\bfM)(\!(u^{-1})\!).\]
        \item For a symplectic vertex, define \[H_i(u):=\frac{u^{\dim \Upsilon_i}W_i(u)\prod_{j-i} V_j(u)(-u)^{\dim \Upsilon_i}W_i(-u)\prod_{j-i} V_j(-u)}{u^2\bbV_i(u-\frac{\hbar}{2})\bbV_i(u+\frac{\hbar}{2})}\in \mathscr{A}^F_\hbar(G,\bfM)(\!(u^{-1})\!).\] 
    \end{enumerate}
\end{definition}
Since these functions are invariant under the action of the Weyl group of $G$, they belong to $\mathscr{A}^F_\hbar(G,\bfM)(\!(u^{-1})\!)$.
\begin{remark}\label{rem:pole}
    If $\vv_i$ is odd, then the denominator in $H_i(u)$ will have $(u-\frac{\hbar}{2})(u+\frac{\hbar}{2})$, which will cancel the corresponding term in the numerator.
\end{remark}

Each $H_i(u)$ expands as a Laurent series in $u^{-1}$, and let $H_{i,r}$ denote the coefficient of $u^{-r-1}$ in $H_i(u)$.
By a direct computation, we get the following properties of $H_i(u)$.
\begin{lemma}\label{lem:Hicond}
    The $H_i(u)$ satisfies $H_i(u)=H_i(-u)$, and $H_{i,r}=0$ if $r< 2\vv_i-\sum_{k-i}\vv_k-\ww_i-3$ for an orthogonal vertex, while $H_{i,r}=0$ if $r< 2\vv_i-\sum_{k-i}\vv_k-\ww_i+1$ for a symplectic vertex.
\end{lemma}

\section{Drinfeld new presentation for the shifted twisted Yangian}\label{sec:yangian}
Now let us fix the underlying graph of the orthosymplectic quiver $Q$ to be the ADE Dynkin diagrams. Let $\mathfrak{g}$ be the corresponding Lie algebra with Cartan matrix $C = (c_{ij})_{1\leq i,j \leq n}$. Let $\mu$ be an even coweight, i.e., $\langle\mu, \alpha_i\rangle\in 2\mathbb{Z}$ for any simple root $\alpha_i$. We recall the Drinfeld new presentations for the twisted Yangian and affine iquantum groups of split type.

\subsection{Twisted Yangian of split type}
We write $[x,y]:=xy-yx$ and $\{x,y\}:=xy+yx$.
\begin{definition} (\cite[Def. 2.2]{LWW})
\label{shiftedTY}
The shifted twisted Yangian ${}^{\iota}\mathbf{Y}_{\mu}(\mathfrak{g})$ of split type is the $\mathbb{C}(\hbar)$-algebra generated by 
$h_{i,r}$ and $b_{i,s}$ for $i\in I, r\geq -\langle \mu, \alpha_i\rangle-1$ and $s\geq 0$, subject to the following relations:
\begin{align}
\label{eq:hcomm}
    & {\left[h_{i, r}, h_{j, s}\right]=0, \qquad 
    h_{i,2s} = 0}, 
        \\
\label{eq:hbcomm}
    & {\left[h_{i, r+2}, b_{j, s}\right]-\left[h_{i, r}, b_{j, s+2}\right]} 
    =c_{i j} \hbar\left\{h_{i, r}, b_{j, s+1}\right\}
    +\frac{c_{i j}^2}{4} \hbar^2\left[h_{i, r}, b_{j, s}\right], 
        \\
\label{eq:bcomm}
    & {\left[b_{i, r+1}, b_{j, s}\right]-\left[b_{i, r}, b_{j, s+1}\right]=\frac{c_{i j}}{2}\hbar \left\{b_{i, r}, b_{j, s}\right\}-2 \delta_{ i, j}(-1)^r h_{j, r+s+1}},
\end{align}
and the Serre-type relations  
\begin{itemize}
    \item when $c_{i j}=0$, we have 
\begin{equation}\label{eq:commSerre}
    \left[b_{i, r}, b_{j, s}\right]=0;
\end{equation}
\item when $c_{i, j}=-1$, we have 
\begin{equation}\label{eq:iSerre}
    \operatorname{Sym}_{k_1, k_2}\left[b_{i, k_1},\left[b_{i, k_2}, b_{j, r}\right]\right]= (-1)^{k_1} [h_{i,k_1+k_2+1},b_{j,r-1}].
\end{equation}
\end{itemize}
Here we have adopted the convention that 
\begin{equation}\label{eq:vanishing}
\mbox{$h_{i,r}=0$ if $r<-\langle \mu,\alpha_i\rangle-1$}.
\end{equation}
Moreover, if $r=0$ and $c_{i,j}=-1$, we use the following convention in \eqref{eq:iSerre}:
\[
[h_{i,r},b_{j,-1}]
:=
\sum_{p\geq0}2^{-2p}\hbar^{2p}
\left(
[h_{i,r-2p-2},b_{j,1}]
-
\hbar\{h_{i,r-2p-2},b_{j,0}\}
\right).
\]
\end{definition}

\begin{remark}
Our convention in Def.~\ref{shiftedTY} agrees with that of \cite{LZ24}.
It differs from the convention in \cite{LWW} only by a shift of the
indices: our \(h_{i,r}\) and \(b_{i,r}\) correspond to their
\(H_i^{(r+1)}\) and \(B_i^{(r+1)}\), respectively. Thus the exceptional
mode \(b_{j,-1}\) here corresponds to \(B_j^{(0)}\) in the notation of
\cite{LWW}.
\end{remark}

We define the following generating functions
\begin{align}
\label{generatingfunction}
h_i(z) & = \hbar\sum_{r\in \mathbb{Z}}h_{i,r}z^{-r-1},&
h_i^\circ(z) & = \hbar\sum_{r\geq 0}h_{i,r}z^{-r-1},&
b_i(z) & = \hbar\sum_{s\geq 0}b_{i,s}z^{-s-1}.
\end{align}
Following \cite[\S 3.3]{LZ24}, we can rewrite the relations as 
\begin{align}
\label{eq:hcomm(z)}&
{\left[h_i(u), h_j(v)\right]=0, \qquad h_{i}(u)=h_i(-u),}\\
\label{eq:hbcomm(z)}
&\left(u^2-v^2\right)\left[h_i(u), b_j(v)\right]= c_{i j} \hbar v\left\{h_i(u), b_j(v)\right\} +\frac{c_{i j}^2}{4} \hbar^2\left[h_i(u), b_j(v)\right]
\\\notag&\qquad 
-\hbar\left[h_i(u), b_{j, 1}\right] 
-\hbar v\left[h_i(u), b_{j, 0}\right]-c_{i j} \hbar^2\left\{h_i(u), b_{j, 0}\right\},\\
\label{eq:bcomm(z)}&
(u-v)\left[b_i(u), b_j(v)\right]=\frac{c_{i j}}{2} \hbar\left\{b_i(u), b_j(v)\right\}+\hbar\left(\left[b_{i, 0}, b_j(v)\right]-\left[b_i(u), b_{j, 0}\right]\right) \\
\notag&\qquad -\delta_{ i, j} \hbar\left(\frac{2 u}{u+v} h_i^{\circ}(u)+\frac{2 v}{u+v} h_j^{\circ}(v)\right),\\
\label{eq:commSerre(z)}&
\left[b_i(u), b_j(v)\right]=0, \textit{ if } c_{ij}=0,
\end{align} 
and the following Serre relation when $c_{ij}=-1$:
\begin{align}\label{eq:usualSerre(z)}
    &\operatorname{Sym}_{u_1, u_2}\left[b_i(u_1),\left[b_i(u_2), b_j(v)\right]\right]\\
    =&\frac{4\hbar}{u_1+u_2}\operatorname{Sym}_{u_1, u_2}\frac{u_2(v-\hbar)h_i(u_2)b_j(v)-u_2(v+\hbar)b_j(v)h_i(u_2)}{4u_2^2-\hbar^2}.\notag
\end{align}

\begin{lemma}\label{lem:serre}
Assume that $c_{ij}=-1$. Then \eqref{eq:usualSerre(z)} is equivalent to
\begin{equation}\label{eq:Serre-special-consequence}
\begin{aligned}
[b_i(u),[b_{i,0},b_{j,0}]]
+
[b_{i,0},[b_i(u),b_{j,0}]]
 =
\left(
b_j(-u+\hbar/2)h_i(u)
-
b_j(u+\hbar/2)h_i(u)
\right)^\circ.
\end{aligned}
\end{equation}
\end{lemma}

\begin{proof}
We first show that \eqref{eq:usualSerre(z)} implies \eqref{eq:Serre-special-consequence}. 
For a Laurent series
\(
A(t)=\sum_{m\in\mathbb{Z}}A_m t^m,
\)
we write
\[
[t^m]A(t):=A_m.
\]

Set \(u_1=u\) and compare the
coefficient of \(u_2^{-1}v^{-1}\). Since
\[
b_i(u_2)=\hbar b_{i,0}u_2^{-1}+O(u_2^{-2}),
\qquad
b_j(v)=\hbar b_{j,0}v^{-1}+\hbar b_{j,1}v^{-2}+O(v^{-3}),
\]
the left-hand side contributes
\[
\hbar^2
\left(
[b_i(u),[b_{i,0},b_{j,0}]]
+
[b_{i,0},[b_i(u),b_{j,0}]]
\right).
\]
Moreover, we have
\[
[v^{-1}]\bigl((v\pm\hbar)b_j(v)\bigr)
=
\hbar b_{j,1}\pm\hbar^2b_{j,0},
\]

Now we set
\[
F(u):=
\frac{4u}{4u^2-\hbar^2}
\left(
[h_i(u),b_{j,1}]
-
\hbar\{h_i(u),b_{j,0}\}
\right).
\]
taking the
coefficient of \(v^{-1}\) on the right-hand side of
\eqref{eq:usualSerre(z)} gives
\[
\hbar^2\frac{F(u_2)+F(u)}{u_2+u}.
\]
Comparing with the
left-hand side, we obtain
\begin{equation*}
[b_i(u),[b_{i,0},b_{j,0}]]
+
[b_{i,0},[b_i(u),b_{j,0}]]
 =
[u_2^{-1}]
\frac{F(u_2)+F(u)}{u_2+u}.
\end{equation*}

Notice that \(F(-u)=-F(u)\), since \(h_i(-u)=h_i(u)\).
Furthermore, \(F(u)\) has only finitely many nonnegative powers.
For any such odd Laurent series, one shows that
\begin{equation*}
[u_2^{-1}]
\frac{F(u_2)+F(u)}{u_2+u}
=
F(u)^\circ.
\end{equation*}

Therefore, we conclude that
\begin{equation}\label{eq:coefficient-serre-reduction-truncated}
\begin{aligned}
&[b_i(u),[b_{i,0},b_{j,0}]]
+
[b_{i,0},[b_i(u),b_{j,0}]]
\\
&\qquad =
\left(
\frac{4u}{4u^2-\hbar^2}
\left(
[h_i(u),b_{j,1}]
-
\hbar\{h_i(u),b_{j,0}\}
\right)
\right)^\circ.
\end{aligned}
\end{equation}

We now use \eqref{eq:hbcomm(z)}. Since \(c_{ij}=-1\), it gives
\[
\begin{aligned}
(u^2-v^2)[h_i(u),b_j(v)]
&=
-\hbar v\{h_i(u),b_j(v)\}
+\frac{\hbar^2}{4}[h_i(u),b_j(v)]
\\
&\qquad
-\hbar[h_i(u),b_{j,1}]
-\hbar v[h_i(u),b_{j,0}]
+\hbar^2\{h_i(u),b_{j,0}\}.
\end{aligned}
\]
Rearranging, we obtain
\begin{equation}\label{eq:hb-rearranged-for-bj1}
\begin{aligned}
&\hbar
\left(
[h_i(u),b_{j,1}]
-\hbar\{h_i(u),b_{j,0}\}
\right)
\\
&\quad =
-\left(u^2-v^2-\frac{\hbar^2}{4}\right)[h_i(u),b_j(v)]
-\hbar v\{h_i(u),b_j(v)\}
-\hbar v[h_i(u),b_{j,0}]
\\
&\quad =
-\left(u^2-(v-\hbar/2)^2\right)h_i(u)b_j(v)
+\left(u^2-(v+\hbar/2)^2\right)b_j(v)h_i(u)
-\hbar v[h_i(u),b_{j,0}].
\end{aligned}
\end{equation}

Putting \(v=-u+\hbar/2\) in \eqref{eq:hb-rearranged-for-bj1}, we have
\begin{equation}\label{eq:negative-shift-reduction}
\begin{aligned}
[h_i(u),b_{j,1}]
-\hbar\{h_i(u),b_{j,0}\}
=
(2u-\hbar)b_j(-u+\hbar/2)h_i(u)
+
(u-\hbar/2)[h_i(u),b_{j,0}].
\end{aligned}
\end{equation}
Putting \(v=u+\hbar/2\) in \eqref{eq:hb-rearranged-for-bj1}, we have
\begin{equation}\label{eq:positive-shift-reduction}
\begin{aligned}
[h_i(u),b_{j,1}]
-\hbar\{h_i(u),b_{j,0}\}
=
-(2u+\hbar)b_j(u+\hbar/2)h_i(u)
-
(u+\hbar/2)[h_i(u),b_{j,0}].
\end{aligned}
\end{equation}

Dividing \eqref{eq:negative-shift-reduction} by \(2u-\hbar\)
and \eqref{eq:positive-shift-reduction} by \(2u+\hbar\), and then
adding the resulting identities, we obtain
\[
\begin{aligned}
&\frac{4u}{4u^2-\hbar^2}
\left(
[h_i(u),b_{j,1}]
-\hbar\{h_i(u),b_{j,0}\}
\right)
\\
&\qquad =
b_j(-u+\hbar/2)h_i(u)
-
b_j(u+\hbar/2)h_i(u).
\end{aligned}
\]
Combining this with
\eqref{eq:coefficient-serre-reduction-truncated} gives
\eqref{eq:Serre-special-consequence}.

Conversely, assume \eqref{eq:Serre-special-consequence}.
By the computation above, using only \eqref{eq:hbcomm(z)}, its
right-hand side can be rewritten as
\[
\left(
\frac{4u}{4u^2-\hbar^2}
\left(
[h_i(u),b_{j,1}]
-\hbar\{h_i(u),b_{j,0}\}
\right)
\right)^\circ.
\]
We first recover the \(k_1=k_2=r=0\) case of
\eqref{eq:iSerre}. Taking the coefficient of \(u^{-1}\) on the
left-hand side of \eqref{eq:Serre-special-consequence} gives
\[
2\hbar[b_{i,0},[b_{i,0},b_{j,0}]].
\]
On the other hand,
\[
\frac{4u}{4u^2-\hbar^2}
=
\sum_{p\geq0}2^{-2p}\hbar^{2p}u^{-2p-1},
\]
so the coefficient of \(u^{-1}\) on the right-hand side is
\[
\hbar
\sum_{p\geq0}2^{-2p}\hbar^{2p}
\left(
[h_{i,-2p-1},b_{j,1}]
-\hbar\{h_{i,-2p-1},b_{j,0}\}
\right).
\]
By our convention for \([h_{i,r},b_{j,-1}]\), the sum above is
precisely \([h_{i,1},b_{j,-1}]\). Therefore,
\[
2[b_{i,0},[b_{i,0},b_{j,0}]]
=
[h_{i,1},b_{j,-1}],
\]
which is exactly the \(k_1=k_2=r=0\) case of
\eqref{eq:iSerre}.

Starting from this relation, the recursive argument of
\cite{LZ24}, using \eqref{eq:hbcomm} and \eqref{eq:bcomm},
recovers \eqref{eq:iSerre} for arbitrary \(k_1,k_2,r\).
Consequently, \eqref{eq:usualSerre(z)} follows. A similar
reduction from a distinguished finite Serre relation to the full
family of Serre relations appears in
\cite[Proposition~6.3]{SSX26}.
\end{proof}

Now take 
\[\mu:=\sum_{i}(\ww_i\Lambda_i-\vv_i\alpha_i^\vee+2\epsilon_i\Lambda_i),\]
where $\Lambda_i$ is the fundamental coweight, and $\alpha_i^\vee$ is the simple coroot. Here we identify the values $\{\pm\}$ of $\epsilon_i$ with $\{\pm 1\}$. By the anomaly cancellation condition in Proposition \ref{prop:ano}, $\mu$  is even.
\begin{theorem} \label{thm:gklo}
    The assignment 
        \[b_i(u)\mapsto B_i(u), \quad h_i(u)\mapsto H_i(u)\]
        defines a $H_F^*(\pt)(\hbar)$-algebra homomorphism 
        \[\Phi:{}^{\iota}\mathbf{Y}_{\mu}(\mathfrak{g})\otimes H_F^*(\pt)\rightarrow \Diff_\hbar(T)\otimes_{\mathbb{C}[\hbar]} H_F^*(\pt)(\hbar).\]
\end{theorem}
To prove this theorem, we need to check that $B_i(u)$ and $H_i(u)$ satisfy the defining relations of ${}^{\iota}\mathbf{Y}_{\mu}(\mathfrak{g})$. We will check this in Section \ref{sec:veri} below. 
\begin{remark}\label{rem:gklo}
    The homomorphism of the above type was first studied by Gerasimov--Kharchev--Lebedev--Oblezin (GKLO) \cite{GKLO05} for the shifted Yangian. Later it was generalized in different settings, see \cite{KWWY14,BFN19,FT19,lu2025shifted,bartlett2025gklo,LWW26,li2026gklorepresentationsshiftedquantum,SSX26,wang2025quivers}. Our formulae for $B_i(u)$ and $H_i(u)$ come from geometry, and they are different from the ones in \cite{lu2025shifted,bartlett2025gklo} in the following ways: 
    \begin{itemize}
        \item The formulas in \cite[Theorem 7.2]{bartlett2025gklo} do not incorporate the equivariant variables from the flavor symmetry group.
        \item The denominators in our operators $B_i(u)$ in Proposition \ref{prop:Bi} are given by the Euler class of the tangent space at torus fixed points of the Schubert cell, while the ones in \cite[Theorem 3.6]{lu2025shifted} have an extra factor. The numerators and the constant term for the symplectic vertex are genuinely different.
    \end{itemize}
\end{remark}

\begin{theorem}\label{thm:main}
   The homomorphism $\Phi$ in Theorem \ref{thm:gklo} factors through a unique $H_{F}^*(\pt)(\hbar)$-algebra homomorphism $\Psi$:
   \[\xymatrix{{}^{\iota}\mathbf{Y}_{\mu}(\mathfrak{g})\otimes H_{F}^*(\pt) \ar[rr]^-\Psi \ar[rd]^-\Phi & &\mathscr{A}^F_\hbar(G,\bfM)\otimes_{\mathbb{C}[\hbar]}\mathbb{C}(\hbar) \ar@{^{(}->}[ld]
   \\
   &\Diff_\hbar(T)\otimes_{\mathbb{C}[\hbar]} H_{F}^*(\pt)(\hbar),}\] 
   such that for any $i\in Q_0$, $\Psi(h_{i,r})=\hbar^{-1}H_{i,r}$, and
   \[\Psi(b_{i,r})=\begin{cases} \hbar^{-1}(-1)^{1+w_i} (-c_1(\mathcal{L}_i)-\tfrac{\hbar}{2})^r\cap [\calR_{G_+,\bfN_+,\epsilon_{i,1}}] & \textit{ if }\epsilon(i)=+;\\
       (-1)^{1+\frac{\ww_i+\sum_{j-i}\vv_j}{2}}2\hbar^{-1}\tilde{p}_*\big((-c_1(\mathcal{L}_i)-\tfrac{\hbar}{2})^r\cap [\widetilde{\calR}_{G_-,\bfN_-,\leq \epsilon_{i,1}}]\big)+\cdots & \textit{ if } \epsilon(i)=-,
   \end{cases}\]
   where $\cdots$ is some element in $H_{G_-\times F\times \mathbb{C}^\times_\hbar}^*(\pt)[\hbar^{-1}]$.
\end{theorem}
\begin{proof}
    This follows from Theorem \ref{thm:gklo}, Proposition \ref{prop:orth}, Proposition \ref{prop:symp}, and Proposition \ref{prop:Bi}.
\end{proof}

\section{Verification of relations}\label{sec:veri}
The goal of this section is to prove Theorem \ref{thm:gklo}. 
The method is not much different from that of \cite{LWW,bartlett2025gklo}. 
We include the verification here for completeness. 

Recall that 
\[H_i(u)=\begin{cases}
        (-1)^{\vv_i+\sum_{k-i}v_k}\frac{(4u^2-\hbar^2)W_i(u)W_i(-u)\prod_{k-i}\bbV_k(u)}{\bbV_i(u-\frac{\hbar}{2})\bbV_i(u+\frac{\hbar}{2})}, & \textit{ if } \epsilon(i)=+;\\
        \frac{u^{\dim \Upsilon_i}W_i(u)\prod_{j-i} V_j(u)(-u)^{\dim \Upsilon_i}W_i(-u)\prod_{j-i} V_j(-u)}{u^2\bbV_i(u-\frac{\hbar}{2})\bbV_i(u+\frac{\hbar}{2})}, & \textit{ if } \epsilon(i)=-.
    \end{cases}\]
Moreover, 
\[B_i(u)
        =\sum_{r=1}^{v_i}\frac{1}{u+\frac{\hbar}{2}+x_{i,r}}\xi_{i,r}^+
        +\sum_{r=1}^{v_i}\frac{1}{u+\frac{\hbar}{2}-x_{i,r}}\xi_{i,r}^-+
        \frac{\xi_i}{u},\]
where
\[\xi_{i,r}^+=\begin{cases}
            -\frac{W_i(x_{i,r}+\frac{\hbar}{2})}{\bbV_{i,r}(x_{i,r})/2x_{i,r}}d_{i,r}, & \textit{ if } \epsilon(i)=+ ;\\  
            -\frac{(x_{i,r}+\frac{\hbar}{2})^{\dim \Upsilon_i}W_i(x_{i,r}+\frac{\hbar}{2})\prod_{j-i}V_j(x_{i,r}+\frac{\hbar}{2})}{(x_{i,r}+\frac{\hbar}{2})\bbV_{i,r}(x_{i,r})}d_{i,r}, & \textit{ if } \epsilon(i)=-;
        \end{cases}\]        
\[\xi_{i,r}^-=\begin{cases}
            -(-1)^{\vv_i+\sum_{k-i}v_k}\frac{W_i(-x_{i,r}+\frac{\hbar}{2})\prod_{k-i}\bbV_k(x_{i,r}-\frac{\hbar}{2})}{\bbV_{i,r}(x_{i,r})/2x_{i,r}}d_{i,r}^{-1}, & \textit{ if } \epsilon(i)=+ ;\\   
            -\frac{(-x_{i,r}+\frac{\hbar}{2})^{\dim \Upsilon_i}W_i(-x_{i,r}+\frac{\hbar}{2})\prod_{j-i}V_j(-x_{i,r}+\frac{\hbar}{2})}{(x_{i,r}-\frac{\hbar}{2})\bbV_{i,r}(x_{i,r})}d_{i,r}^{-1}, & \textit{ if } \epsilon(i)=-,
        \end{cases},\]
and 
\[\xi_i=\begin{cases}
        0, &\textit{ if } \epsilon(i)=+;\\
        -\delta_{\dim \Upsilon_i,0}\frac{W_i(0)\prod_{j-i}V_j(0)}{\bbV_i(\frac{\hbar}{2})}, & \textit{ if }\epsilon(i)=-.
    \end{cases}\]

For convenience, we let $x_{i,0}:=\frac{\hbar}{2}$, $\xi_{i,0}^-:=\xi_i$ and $\xi_{i,0}^+=0$. Hence, 
\[B_i(u)
        =\sum_{r=0}^{v_i}\frac{1}{u+\frac{\hbar}{2}+x_{i,r}}\xi_{i,r}^+
        +\sum_{r=0}^{v_i}\frac{1}{u+\frac{\hbar}{2}-x_{i,r}}\xi_{i,r}^-.\]

First of all, by Lemma \ref{lem:Hicond} and the definition of $\mu$, we get the relation \eqref{eq:hcomm(z)}, and the compatible vanishing of $H_{i,r}$ and $h_{i,r}$. Since $B_i(u)$ only involves shift operators $d_{i,r}$, the relation \eqref{eq:commSerre(z)} is obvious.  Let us verify the remaining relations. Since $b_i(u)\mapsto B_i(u)$, $\hbar b_{i,r}$ is sent to $B_{i,r}$, the coefficient of $u^{-r-1}$ in $B_i(u)$.

\subsection{Some preparations}
We need some combinatorial identities.
For a Laurent series $f(z)=\sum_{i\in \mathbb{Z}} a_iz^{i}$ in $z^{-1}$ ,
we denote its truncation 
\[\big(f(z)\big)^{\circ}
:=\sum_{i< 0}a_iz^i.\] 
If $\tilde{f}(z):=f(z)/z$, then
\begin{equation}\label{equ:divz}
    u\big(\tilde{f}(u)\big)^\circ+v\big(\tilde{f}(-v)\big)^\circ
    =u\sum_{i< 1}a_iu^{i-1}-(-v)\sum_{i< 1}a_i(-v)^{i-1}=\big(f(u)\big)^{\circ}-\big(f(-v)\big)^{\circ}.
\end{equation}

\begin{lemma}[{\cite[Lemma 5.1]{SSX26}}]
\label{lem:partialfrac}
    Assume $f(z)=\dfrac{g(z)}{\prod_{i=1}^n (z-z_i)}$, where $z_1,\ldots,z_n$ are distinct and $g(z)\in \mathbb{C}[z]$. Then
    \[(f(z))^\circ=\sum_{i=1}^n\frac{1}{z-z_i}\Res(f(z),z_i).\]
\end{lemma} 

Recall $d_{i,r}x_{j,s}=(x_{j,s}+\delta_{i,j}\delta_{r,s}\hbar)d_{i,r}$ for $1\leq r\leq v_i, 1\leq s\leq v_j$.
By direct computations, we get the following commutation relations.
\begin{lemma}\label{lem:xii}
    For $0\leq r\neq s\leq v_i$,
    \[\xi_{i,s}^\pm\xi_{i,r}^\pm=\frac{\pm x_{i,r}\mp x_{i,s}+\hbar}{\pm x_{i,r}\mp x_{i,s}-\hbar}\xi_{i,r}^\pm\xi_{i,s}^\pm,\]
    and
    \[\xi_{i,s}^\mp\xi_{i,r}^\pm=\frac{\pm x_{i,r}\pm x_{i,s}+\hbar}{\pm x_{i,r}\pm x_{i,s}-\hbar}\xi_{i,r}^\pm\xi_{i,s}^\mp.\]
\end{lemma}

\begin{lemma}\label{lem:xipm}
    Suppose $c_{ij}=-1$. For any $0\leq r\leq v_i$ and $0\leq s\leq v_j$, we have
    \[\xi_{j,s}^\pm\xi_{i,r}^\pm=\frac{\pm x_{i,r}-\frac{\hbar}{2}\mp x_{j,s}}{\pm x_{i,r}+\frac{\hbar}{2}\mp x_{j,s}}\xi_{i,r}^\pm\xi_{j,s}^\pm,\]
    and
    \[\xi_{j,s}^\pm\xi_{i,r}^\mp=\frac{\pm x_{i,r}+\frac{\hbar}{2}\pm x_{j,s}}{\pm x_{i,r}-\frac{\hbar}{2}\pm x_{j,s}}\xi_{i,r}^\mp\xi_{j,s}^\pm.\]
\end{lemma}   
\begin{lemma}\label{lem:xih}
    \begin{enumerate}
        \item For each vertex $i$ and $0\leq r\leq v_i$,
        \[\xi_{i,r}^\pm H_i(u)=\frac{(u-\frac{\hbar}{2}\pm x_{i,r})(u+\frac{\hbar}{2}\mp x_{i,r})}{(u-\frac{3\hbar}{2}\mp x_{i,r})(u+\frac{3\hbar}{2}\pm x_{i,r})}H_i(u)\xi_{i,r}^\pm .\]
        \item For two vertices $i,j$ satisfying $c_{ij}=-1$, and $1\leq r\leq v_j$,
        \[\xi_{j,r}^\pm H_i(u)=\frac{(u\mp x_{j,r}-\hbar)(u\pm x_{j,r}+\hbar)}{(u\mp x_{j,r})(u\pm x_{j,r})}H_i(u)\xi_{j,r}^\pm.\]
    \end{enumerate}   
\end{lemma}

Applying Lemma \ref{lem:partialfrac} to $H_i(u)$, we get the following formula.
\begin{lemma}\label{lem:H_i}
    For any vertex $i$,
    \begin{align*}
        -2\hbar(uH_i(u))^\circ=&\sum_{r=1}^{v_i}\bigg(\big(\frac{-2\hbar-2x_{i,r}}{u+\frac{\hbar}{2}+x_{i,r}}+\frac{-2\hbar-2x_{i,r}}{u-\frac{\hbar}{2}-x_{i,r}}\big)\xi^+_{i,r}\xi^-_{i,r}\\
        &+\big(\frac{-2\hbar+2x_{i,r}}{u+\frac{\hbar}{2}-x_{i,r}}+\frac{-2\hbar+2x_{i,r}}{u-\frac{\hbar}{2}+x_{i,r}}\big)\xi^-_{i,r}\xi^+_{i,r}\bigg)-\frac{2\hbar \xi_i^2}{u}.
    \end{align*}
\end{lemma}
\begin{proof}
    By the definition of $H_i(u)$,
    $uH_i(u)$ is an odd function. By Remark \ref{rem:pole}, $uH_i(u)$ has simple poles at $\{\pm(x_{i,r}+\frac{\hbar}{2}), \pm(x_{i,r}-\frac{\hbar}{2})\mid 1\leq r\leq v_i\}$ and also at $u=0$ if $\epsilon(i)=-$ and $\dim \Upsilon_i=0$. By a direct computation,
    \[\Res\big(-2\hbar uH_i(u), x_{i,r}+\frac{\hbar}{2}\big)=\Res\big(-2\hbar uH_i(u), -x_{i,r}-\frac{\hbar}{2}\big)=(-2\hbar-2x_{i,r})\xi_{i,r}^+\xi_{i,r}^-,\]
    and
    \[\Res\big(-2\hbar uH_i(u), x_{i,r}-\frac{\hbar}{2}\big)=\Res\big(-2\hbar uH_i(u), -x_{i,r}+\frac{\hbar}{2}\big)=(-2\hbar+2x_{i,r})\xi_{i,r}^-\xi_{i,r}^+.\]
    If $\epsilon(i)=-$,
    \[\Res\big(-2\hbar uH_i(u), 0\big)=-2\hbar\frac{\delta_{\dim \Upsilon_i,0}W_i(0)\prod_{j-i} V_j(0)W_i(0)\prod_{j-i} V_j(0)}{\bbV_i(-\frac{\hbar}{2})\bbV_i(\frac{\hbar}{2})}=-2\hbar \xi_i^2.\]
    Then the lemma follows from Lemma \ref{lem:partialfrac}.
\end{proof}

\subsection{Relation \eqref{eq:hbcomm(z)}}
If $c_{ij}=0$, then it is easy to see that $\left[H_i(u), B_j(v)\right]=0$. Thus, the relation holds in this case. Now let us assume $c_{ij}\neq 0$.

\subsubsection{Case $j=i$}\label{sec:HiBi1}
We need to show
\begin{align}\label{equ:HiBi}
    &\left(u^2-v^2\right)\left[H_i(u), B_i(v)\right]- 2\hbar v\left\{H_i(u), B_i(v)\right\} - \hbar^2\left[H_i(u), B_i(v)\right]\\
    &+\left[H_i(u), B_{i, 1}\right] +v\left[H_i(u), B_{i, 0}\right]+2 \hbar\left\{H_i(u), B_{i,0}\right\}=0.\notag
\end{align}
By Lemma \ref{lem:xih}(1), the coefficient of $H_i(u)\xi_{i,r}^+$ for $1\leq r\leq v_i$ on the left hand side of \eqref{equ:HiBi} is 
\begin{align*}
    &\frac{(u^2-v^2)-2\hbar v-\hbar^2}{v+\frac{\hbar}{2}+x_{i,r}}-\frac{(u^2-v^2)+2\hbar v-\hbar^2}{v+\frac{\hbar}{2}+x_{i,r}}\frac{(u-\frac{\hbar}{2}+x_{i,r})(u+\frac{\hbar}{2}-x_{i,r})}{(u-\frac{3\hbar}{2}-x_{i,r})(u+\frac{3\hbar}{2}+x_{i,r})}\\
    &-(\frac{\hbar}{2}+x_{i,r})\bigg(1-\frac{(u-\frac{\hbar}{2}+x_{i,r})(u+\frac{\hbar}{2}-x_{i,r})}{(u-\frac{3\hbar}{2}-x_{i,r})(u+\frac{3\hbar}{2}+x_{i,r})}\bigg)\\
    &+v+2\hbar-(v-2\hbar)\frac{(u-\frac{\hbar}{2}+x_{i,r})(u+\frac{\hbar}{2}-x_{i,r})}{(u-\frac{3\hbar}{2}-x_{i,r})(u+\frac{3\hbar}{2}+x_{i,r})}=0.
\end{align*}
Similarly, the coefficient of $H_i(u)\xi_{i,r}^-$ is also zero. If $\epsilon(i)=-$, it is direct to see that the coefficient of $\xi_iH_i(u)$ is also zero. Hence, \eqref{equ:HiBi} holds in this case. 

\subsubsection{Case $c_{ij}=-1$}
We need to show
\begin{align}\label{equ:HiBj1}
    &\left(u^2-v^2\right)\left[H_i(u), B_j(v)\right]+\hbar v\left\{H_i(u), B_j(v)\right\} -\frac{1}{4} \hbar^2\left[H_i(u), B_j(v)\right]\\
    &+\left[H_i(u), B_{j, 1}\right] + v\left[H_i(u), B_{j, 0}\right]-\hbar\left\{H_i(u), B_{j, 0}\right\}=0.\notag
\end{align}

By Lemma \ref{lem:xih}(2), the coefficient of $H_i(u)\xi_{j,r}^+$ for $1\leq r\leq v_j$ on the left hand side of \eqref{equ:HiBj1} is
\begin{align*}
    &\frac{(u^2-v^2)+\hbar v-\frac{1}{4}\hbar^2}{v+\frac{\hbar}{2}+x_{j,r}}-\frac{(u^2-v^2)-\hbar v-\frac{1}{4}\hbar^2}{v+\frac{\hbar}{2}+x_{j,r}}\frac{(u-x_{j,r}-\hbar)(u+x_{j,r}+\hbar)}{(u-x_{j,r})(u+x_{j,r})}\\
    &-(\frac{\hbar}{2}+x_{j,r})\bigg(1-\frac{(u-x_{j,r}-\hbar)(u+x_{j,r}+\hbar)}{(u-x_{j,r})(u+x_{j,r})}\bigg)\\
    &+v-\hbar-(v+\hbar)\frac{(u-x_{j,r}-\hbar)(u+x_{j,r}+\hbar)}{(u-x_{j,r})(u+x_{j,r})}=0.
\end{align*}
Similarly, the coefficient of $H_i(u)\xi_{j,r}^-$ is also zero by Lemma \ref{lem:xih}(2). If $\epsilon(j)=-$, the coefficient of $\xi_jH_i(u)$ is obviously $0$. This concludes the verification of Relation \eqref{eq:hbcomm(z)}.

\subsection{Relation \eqref{eq:bcomm(z)}}

\subsubsection{Case $j\neq i$}
If $c_{ij}=0$, then it is obvious that $\left[B_i(u), B_j(v)\right]=0$. Hence, the relation holds in this case. Now let us assume $c_{ij}=-1$, i.e. $i$ and $j$ are connected by an edge. Without loss of generality, let us assume $\epsilon(i)=-$.
We need to show 
\begin{equation}\label{equ:BBii1}
    (u-v)\left[B_i(u), B_j(v)\right]+\frac{1}{2} \hbar\left\{B_i(u), B_j(v)\right\}-\big(\left[B_{i, 0}, B_j(v)\right]-\left[B_i(u), B_{j, 0}\right]\big)=0.
\end{equation}
By definition,
\begin{align*}
    B_i(u)B_j(v)=&\big(\sum_{r=1}^{v_i}\frac{1}{u+\frac{\hbar}{2}+x_{i,r}}\xi_{i,r}^+
        +\sum_{r=1}^{v_i}\frac{1}{u+\frac{\hbar}{2}-x_{i,r}}\xi_{i,r}^-+
        \frac{\xi_i}{u}\big)\\
        &\big(\sum_{s=1}^{v_j} \frac{1}{v+\frac{\hbar}{2}+x_{j,s}}\xi^+_{j,s}+\sum_{s=1}^{v_j} \frac{1}{v+\frac{\hbar}{2}-x_{j,s}}\xi^-_{j,s}\big).
\end{align*}
Therefore, by Lemma \ref{lem:xipm}, the coefficient of $\xi_{i,r}^+\xi_{j,s}^+$ in the left hand side of \eqref{equ:BBii1} is
\begin{align*}
    &(u-v+\frac{\hbar}{2})\frac{1}{u+\frac{\hbar}{2}+x_{i,r}}\frac{1}{v+\frac{\hbar}{2}+x_{j,s}}-\frac{1}{v+\frac{\hbar}{2}+x_{j,s}}+\frac{1}{u+\frac{\hbar}{2}+x_{i,r}}\\
    &-\bigg((u-v-\frac{\hbar}{2})\frac{1}{u+\frac{\hbar}{2}+x_{i,r}}\frac{1}{v+\frac{\hbar}{2}+x_{j,s}}-\frac{1}{v+\frac{\hbar}{2}+x_{j,s}}+\frac{1}{u+\frac{\hbar}{2}+x_{i,r}}\bigg)\frac{x_{i,r}-\frac{\hbar}{2}-x_{j,s}}{x_{i,r}+\frac{\hbar}{2}-x_{j,s}}=0.
\end{align*}
A similar calculation shows that the other coefficients are 0.

\subsubsection{Case $j=i$ and $\epsilon(i)=+$}\label{sec:j=iorth}
We need to show
\begin{align}\label{equ:BB}
    &(u-v)\left[B_i(u), B_i(v)\right]- \hbar\left\{B_i(u), B_i(v)\right\}-\left(\left[B_{i, 0}, B_i(v)\right]-\left[B_i(u), B_{i, 0}\right]\right) \\
    =&-\hbar\left(\frac{2 u}{u+v} H_i^{\circ}(u)+\frac{2 v}{u+v} H_i^{\circ}(v)\right).\notag
\end{align}

From the definition,
\begin{align*}
    &B_i(u)B_i(v)\\
    =&\sum_{1\leq r\neq s\leq v_i}\bigg(\frac{1}{u+\frac{\hbar}{2}+x_{i,r}}\frac{1}{v+\frac{\hbar}{2}+x_{i,s}}\xi^+_{i,r}\xi^+_{i,s}+\frac{1}{u+\frac{\hbar}{2}+x_{i,r}}\frac{1}{v+\frac{\hbar}{2}-x_{i,s}}\xi^+_{i,r}\xi^-_{i,s}\\
    &+\frac{1}{u+\frac{\hbar}{2}-x_{i,r}}\frac{1}{v+\frac{\hbar}{2}+x_{i,s}}\xi^-_{i,r}\xi^+_{i,s}+\frac{1}{u+\frac{\hbar}{2}-x_{i,r}}\frac{1}{v+\frac{\hbar}{2}-x_{i,s}}\xi^-_{i,r}\xi^-_{i,s}\bigg)\\
    +&\sum_{1\leq r\leq v_i}\bigg(\frac{1}{u+\frac{\hbar}{2}+x_{i,r}}\frac{1}{v+\frac{3\hbar}{2}+x_{i,r}}\xi^+_{i,r}\xi^+_{i,r}+\frac{1}{u+\frac{\hbar}{2}-x_{i,r}}\frac{1}{v+\frac{3\hbar}{2}-x_{i,r}}\xi^-_{i,r}\xi^-_{i,r}\\
    &+\frac{1}{u+\frac{\hbar}{2}+x_{i,r}}\frac{1}{v-\frac{\hbar}{2}-x_{i,r}}\xi^+_{i,r}\xi^-_{i,r}+\frac{1}{u+\frac{\hbar}{2}-x_{i,r}}\frac{1}{v-\frac{\hbar}{2}+x_{i,r}}\xi^-_{i,r}\xi^+_{i,r}\bigg).
\end{align*}
Therefore, by Lemma \ref{lem:xii}, the coefficient of $\xi^+_{i,r}\xi^+_{i,s}$ for $1\leq r\neq s\leq v_i$ in the left hand side of \eqref{equ:BB} is
\begin{align*}
    &(u-v-\hbar)\frac{1}{u+\frac{\hbar}{2}+x_{i,r}}\frac{1}{v+\frac{\hbar}{2}+x_{i,s}}-\frac{1}{v+\frac{\hbar}{2}+x_{i,s}}+\frac{1}{u+\frac{\hbar}{2}+x_{i,r}}\\ 
    &+\bigg(-(u-v+\hbar)\frac{1}{u+\frac{\hbar}{2}+x_{i,r}}\frac{1}{v+\frac{\hbar}{2}+x_{i,s}}-\frac{1}{u+\frac{\hbar}{2}+x_{i,r}}+\frac{1}{v+\frac{\hbar}{2}+x_{i,s}}\bigg)\frac{x_{i,r}-x_{i,s}+\hbar}{x_{i,r}-x_{i,s}-\hbar}=0.
\end{align*}
Similarly, the coefficients of $\xi_{i,r}^+\xi_{i,s}^-$, $\xi_{i,s}^-\xi_{i,r}^+$, and  
$\xi^-_{i,r}\xi^-_{i,s}$ for $1\leq r\neq s\leq v_i$ in the left hand side of \eqref{equ:BB} are all 0.

The coefficient of $\xi_{i,r}^+\xi_{i,r}^+$ is 
\begin{align*}
    &(u-v-\hbar)\frac{1}{u+\frac{\hbar}{2}+x_{i,r}}\frac{1}{v+\frac{3\hbar}{2}+x_{i,r}}-\frac{1}{v+\frac{3\hbar}{2}+x_{i,r}}+\frac{1}{u+\frac{\hbar}{2}+x_{i,r}}\\ 
    &-(u-v+\hbar)\frac{1}{v+\frac{\hbar}{2}+x_{i,r}}\frac{1}{u+\frac{3\hbar}{2}+x_{i,r}}-\frac{1}{u+\frac{3\hbar}{2}+x_{i,r}}+\frac{1}{v+\frac{\hbar}{2}+x_{i,r}}=0.
\end{align*}
So is the coefficient of $\xi_{i,r}^-\xi_{i,r}^-$.

Therefore, the left hand side of \eqref{equ:BB} is
\begin{align*}
    &\sum_{1\leq r\leq v_i}\bigg((u-v-\hbar)\big(\frac{1}{u+\frac{\hbar}{2}+x_{i,r}}\frac{1}{v-\frac{\hbar}{2}-x_{i,r}}\xi^+_{i,r}\xi^-_{i,r}+\frac{1}{u+\frac{\hbar}{2}-x_{i,r}}\frac{1}{v-\frac{\hbar}{2}+x_{i,r}}\xi^-_{i,r}\xi^+_{i,r}\big)\\
    &-\big(\frac{1}{v-\frac{\hbar}{2}-x_{i,r}}\xi^+_{i,r}\xi^-_{i,r}+\frac{1}{v-\frac{\hbar}{2}+x_{i,r}}\xi^-_{i,r}\xi^+_{i,r}\big)+\big(\frac{1}{u+\frac{\hbar}{2}+x_{i,r}}\xi^+_{i,r}\xi^-_{i,r}+\frac{1}{u+\frac{\hbar}{2}-x_{i,r}}\xi^-_{i,r}\xi^+_{i,r}\big)\\
    &-(u-v+\hbar)\big(\frac{1}{v+\frac{\hbar}{2}+x_{i,r}}\frac{1}{u-\frac{\hbar}{2}-x_{i,r}}\xi^+_{i,r}\xi^-_{i,r}+\frac{1}{v+\frac{\hbar}{2}-x_{i,r}}\frac{1}{u-\frac{\hbar}{2}+x_{i,r}}\xi^-_{i,r}\xi^+_{i,r}\big)\\
    &+\big(\frac{1}{v+\frac{\hbar}{2}+x_{i,r}}\xi^+_{i,r}\xi^-_{i,r}+\frac{1}{v+\frac{\hbar}{2}-x_{i,r}}\xi^-_{i,r}\xi^+_{i,r}\big)-\big(\frac{1}{u-\frac{\hbar}{2}-x_{i,r}}\xi^+_{i,r}\xi^-_{i,r}+\frac{1}{u-\frac{\hbar}{2}+x_{i,r}}\xi^-_{i,r}\xi^+_{i,r}\big)\bigg)\\
    =&\frac{1}{u+v}\sum_{1\leq r\leq v_i}\bigg(\big(\frac{-2\hbar-2x_{i,r}}{u+\frac{\hbar}{2}+x_{i,r}}+\frac{-2\hbar-2x_{i,r}}{u-\frac{\hbar}{2}-x_{i,r}}\big)\xi^+_{i,r}\xi^-_{i,r}+\big(\frac{-2\hbar+2x_{i,r}}{u+\frac{\hbar}{2}-x_{i,r}}+\frac{-2\hbar+2x_{i,r}}{u-\frac{\hbar}{2}+x_{i,r}}\big)\xi^-_{i,r}\xi^+_{i,r}\\
    &+\big(\frac{-2\hbar-2x_{i,r}}{v-\frac{\hbar}{2}-x_{i,r}}+\frac{-2\hbar-2x_{i,r}}{v+\frac{\hbar}{2}+x_{i,r}}\big)\xi^+_{i,r}\xi^-_{i,r}+\big(\frac{-2\hbar+2x_{i,r}}{v-\frac{\hbar}{2}+x_{i,r}}+\frac{-2\hbar+2x_{i,r}}{v+\frac{\hbar}{2}-x_{i,r}}\big)\xi^-_{i,r}\xi^+_{i,r}\bigg)\\
    =&\frac{-2\hbar}{u+v}\bigg((uH_i(u))^\circ-(-vH_i(-v))^\circ\bigg)\\
    =&\frac{-2\hbar}{u+v}\big(u(H_i(u))^\circ+v(H_i(v))^\circ\big).
\end{align*}
Here the second equality follows from Lemma \ref{lem:H_i}, while the last one follows from \eqref{equ:divz}. This finishes the proof of \eqref{equ:BB}. Finally, the last case $j=i$ and $\epsilon(i)=-$ can be proved similarly.

\subsection{Serre relation}
We will prove the equivalent Serre relation in Lemma \ref{lem:serre}: if $c_{ij}=-1$, 
\begin{align}\label{equ:serre}
&[B_i(u),[B_{i,0},B_{j,0}]]
+
[B_{i,0},[B_i(u),B_{j,0}]]\\
=&\big(\hbar^2 B_j(-u+\hbar/2)H_i(u)-\hbar^2 B_j(u+\hbar/2)H_i(u)\big)^\circ.\notag
\end{align}

We will frequently use Lemma \ref{lem:xii} and Lemma \ref{lem:xipm} in the computation. 
The left hand side of \eqref{equ:serre} is 
\begin{align}\label{equ:S}
    &\bigg[\sum_{r=0}^{v_i} \frac{1}{u+\frac{\hbar}{2}+x_{i,r}}\xi^+_{i,r}+\sum_{r=0}^{v_i} \frac{1}{u+\frac{\hbar}{2}-x_{i,r}}\xi^-_{i,r},\big[\sum_{s=0}^{v_i}(\xi_{i,s}^++\xi_{i,s}^-),\sum_{t=0}^{v_j}(\xi_{j,t}^++\xi_{j,t}^-)\big]\bigg]\\
    &+\bigg[\sum_{s=0}^{v_i}(\xi_{i,s}^++\xi_{i,s}^-),\big[\sum_{r=0}^{v_i} \frac{1}{u+\frac{\hbar}{2}+x_{i,r}}\xi^+_{i,r}+\sum_{r=0}^{v_i} \frac{1}{u+\frac{\hbar}{2}-x_{i,r}}\xi^-_{i,r},\sum_{t=0}^{v_j}(\xi_{j,t}^++\xi_{j,t}^-)\big]\bigg].\notag
\end{align}
Let us first consider the case of $r\neq s$. Then
\begin{align*}
    &\big[ \xi^+_{i,r},\big[\xi_{i,s}^+,\xi_{j,t}^+\big]\big]+\big[\xi_{i,s}^+,\big[\xi^+_{i,r},\xi_{j,t}^+\big]\big]\\
    =&-\frac{\hbar}{x_{j,t}-\frac{\hbar}{2}-x_{i,s}}\big(\xi_{i,r}^+\xi_{i,s}^+\xi_{j,t}^+-\xi_{i,s}^+\xi_{j,t}^+\xi_{i,r}^+\big)-\frac{\hbar}{x_{j,t}-\frac{\hbar}{2}-x_{i,r}}\big(\xi_{i,s}^+\xi^+_{i,r}\xi_{j,t}^+-\xi^+_{i,r}\xi_{j,t}^+\xi_{i,s}^+)\\
    =&-\bigg(\frac{\hbar}{x_{j,t}-\frac{\hbar}{2}-x_{i,s}}\bigg(1-\frac{(x_{i,r}-x_{i,s}+\hbar)(x_{j,t}-x_{i,r}+\frac{\hbar}{2})}{(x_{i,r}-x_{i,s}-\hbar)(x_{j,t}-x_{i,r}-\frac{\hbar}{2})}\bigg)\\
    &+\frac{\hbar}{x_{j,t}-\frac{\hbar}{2}-x_{i,r}}\big(\frac{x_{i,r}-x_{i,s}+\hbar}{x_{i,r}-x_{i,s}-\hbar}-\frac{x_{j,t}-x_{i,s}+\frac{\hbar}{2}}{x_{j,t}-x_{i,s}-\frac{\hbar}{2}}\big)\bigg)\xi_{i,r}^+\xi_{i,s}^+\xi_{j,t}^+=0.
\end{align*}
Similarly, for $A\in \{\xi_{i,r}^\pm\}$, $B\in \{\xi_{i,s}^\pm\}$, $C\in \{\xi_{j,t}^\pm\}$ such that $r\neq s$, $$\big[ A,\big[B,C\big]\big]+\big[B,\big[A,C\big]\big]=0,$$ as we just need to change the corresponding variables to its negative.

For the case $1\leq r=s\leq v_i$, we have
\begin{align*}
    &\big[\frac{1}{u+\frac{\hbar}{2}+x_{i,r}}\xi^+_{i,r},\big[\xi_{i,r}^+,\xi_{j,t}^+\big]\big]+\big[\xi_{i,r}^+,\big[\frac{1}{u+\frac{\hbar}{2}+x_{i,r}}\xi^+_{i,r},\xi_{j,t}^+\big]\big]\\
    =&-\frac{1}{u+\frac{\hbar}{2}+x_{i,r}}\frac{\hbar}{x_{j,t}-\frac{3\hbar}{2}-x_{i,r}}\xi^+_{i,r}\xi_{i,r}^+\xi_{j,t}^++\frac{1}{u+\frac{3\hbar}{2}+x_{i,r}}\frac{\hbar}{x_{j,t}-\frac{\hbar}{2}-x_{i,r}}\xi_{i,r}^+\xi_{j,t}^+\xi^+_{i,r}\\
    &-\frac{1}{u+\frac{3\hbar}{2}+x_{i,r}}\frac{\hbar}{x_{j,t}-\frac{3\hbar}{2}-x_{i,r}}\xi_{i,r}^+\xi^+_{i,r}\xi_{j,t}^++\frac{1}{u+\frac{\hbar}{2}+x_{i,r}}\frac{\hbar}{x_{j,t}-\frac{\hbar}{2}-x_{i,r}}\xi^+_{i,r}\xi_{j,t}^+\xi_{i,r}^+\\
    =&\bigg(-\frac{1}{u+\frac{\hbar}{2}+x_{i,r}}\frac{\hbar}{x_{j,t}-\frac{3\hbar}{2}-x_{i,r}}+\frac{1}{u+\frac{3\hbar}{2}+x_{i,r}}\frac{\hbar}{x_{j,t}-\frac{\hbar}{2}-x_{i,r}}\frac{x_{j,t}-x_{i,r}-\frac{\hbar}{2}}{x_{j,t}-x_{i,r}-\frac{3\hbar}{2}}\\
    &-\frac{1}{u+\frac{3\hbar}{2}+x_{i,r}}\frac{\hbar}{x_{j,t}-\frac{3\hbar}{2}-x_{i,r}}+\frac{1}{u+\frac{\hbar}{2}+x_{i,r}}\frac{\hbar}{x_{j,t}-\frac{\hbar}{2}-x_{i,r}}\frac{x_{j,t}-x_{i,r}-\frac{\hbar}{2}}{x_{j,t}-x_{i,r}-\frac{3\hbar}{2}}\bigg)\xi^+_{i,r}\xi_{i,r}^+\xi_{j,t}^+\\
    =&0.
\end{align*}
Similarly, for $C\in \{\xi_{j,t}^\pm\}$,
\[\big[\frac{1}{u+\frac{\hbar}{2}\pm x_{i,r}}\xi^\pm_{i,r},\big[\xi_{i,r}^\pm,C\big]\big]+\big[\xi_{i,r}^\pm,\big[\frac{1}{u+\frac{\hbar}{2}\pm x_{i,r}}\xi^\pm_{i,r},C\big]\big]=0.\]

Therefore, the remaining terms in \eqref{equ:S} are
\begin{align*}
    &\sum_{r=1}^{v_i}\sum_{t=0}^{v_j}\bigg(\big[\frac{1}{u+\frac{\hbar}{2}+x_{i,r}}\xi^+_{i,r},\big[\xi_{i,r}^-,\xi_{j,t}^++\xi_{j,t}^-\big]\big]+\big[\frac{1}{u+\frac{\hbar}{2}-x_{i,r}}\xi^-_{i,r},\big[\xi_{i,r}^+,\xi_{j,t}^++\xi_{j,t}^-\big]\big]\\
    &+\big[\xi_{i,r}^+,\big[\frac{1}{u+\frac{\hbar}{2}-x_{i,r}}\xi^-_{i,r},\xi_{j,t}^++\xi_{j,t}^-\big]\big]+\big[\xi_{i,r}^-,\big[\frac{1}{u+\frac{\hbar}{2}+x_{i,r}}\xi^+_{i,r},\xi_{j,t}^++\xi_{j,t}^-\big]\big]\bigg)\\
    &+\sum_{t=0}^{v_j}\bigg(\big[\frac{1}{u}\xi_i,\big[\xi_i,\xi_{j,t}^++\xi_{j,t}^-\big]\big]+\big[\xi_i,\big[\frac{1}{u}\xi_i,\xi_{j,t}^++\xi_{j,t}^-\big]\big]\bigg)
\end{align*}
The terms involving $\xi_{j,t}^+$ is
\begin{align*}
    &\sum_{r=1}^{v_i}\bigg(\big[\frac{1}{u+\frac{\hbar}{2}+x_{i,r}}\xi^+_{i,r},\big[\xi_{i,r}^-,\xi_{j,t}^+\big]\big]+\big[\frac{1}{u+\frac{\hbar}{2}-x_{i,r}}\xi^-_{i,r},\big[\xi_{i,r}^+,\xi_{j,t}^+\big]\big]\\
    &+\big[\xi_{i,r}^+,\big[\frac{1}{u+\frac{\hbar}{2}-x_{i,r}}\xi^-_{i,r},\xi_{j,t}^+\big]\big]+\big[\xi_{i,r}^-,\big[ \frac{1}{u+\frac{\hbar}{2}+x_{i,r}}\xi^+_{i,r},\xi_{j,t}^+\big]\big]\bigg)\\
    &+\big[\frac{1}{u}\xi_i,\big[\xi_i,\xi_{j,t}^+\big]\big]+\big[\xi_i,\big[\frac{1}{u}\xi_i,\xi_{j,t}^+\big]\big]\\
    =&\sum_{r=1}^{v_i}\bigg(\big(\frac{1}{u+\frac{\hbar}{2}+x_{i,r}}+\frac{1}{u-\frac{\hbar}{2}-x_{i,r}}\big)(-2\xi^+_{i,r}\xi_{j,t}^+\xi_{i,r}^-+\xi^+_{i,r}\xi_{i,r}^-\xi_{j,t}^++\xi_{j,t}^+\xi^+_{i,r}\xi_{i,r}^-)\\
    &+\big(\frac{1}{u-\frac{\hbar}{2}+x_{i,r}}+\frac{1}{u+\frac{\hbar}{2}-x_{i,r}}\big)(-2\xi_{i,r}^-\xi_{j,t}^+\xi^+_{i,r}+\xi_{j,t}^+\xi_{i,r}^-\xi^+_{i,r}+\xi_{i,r}^-\xi^+_{i,r}\xi_{j,t}^+)\bigg)\\
    &+\frac{2\hbar^2}{ux_{j,t}^2}\xi_i^2\xi_{j,t}^+\\
    =&\Bigg(\sum_{r=1}^{v_i}\bigg(\frac{\hbar(-2\hbar-2x_{i,r})}{(x_{i,r}+\frac{\hbar}{2})^2-x_{j,t}^2}\big(\frac{1}{u+\frac{\hbar}{2}+x_{i,r}}+\frac{1}{u-\frac{\hbar}{2}-x_{i,r}}\big)\xi^+_{i,r}\xi^-_{i,r}\\
        &+\frac{\hbar(-2\hbar+2x_{i,r})}{(x_{i,r}-\frac{\hbar}{2})^2-x_{j,t}^2}\big(\frac{1}{u+\frac{\hbar}{2}-x_{i,r}}+\frac{1}{u-\frac{\hbar}{2}+x_{i,r}}\big)\xi^-_{i,r}\xi^+_{i,r}\bigg)+\frac{2\hbar^2}{ux_{j,t}^2}\xi_i^2\Bigg)\xi_{j,t}^+.
\end{align*}

Now let us first assume $1\leq t\leq v_j$, then the terms on the right hand side of \eqref{equ:serre} involving $\xi_{j,t}^+$ are,
\begin{align*}
    &-\frac{\hbar^2}{u-\hbar-x_{j,t}}\xi^+_{j,t}H_i(u)-\frac{\hbar^2}{u+\hbar+x_{j,t}}\xi^+_{j,t}H_i(u)\\
    =&\frac{-2\hbar^2u}{(u-x_{j,t})(u+x_{j,t})}H_i(u)\xi^+_{j,t}.
\end{align*}
Here we used Lemma \ref{lem:xih}.
By the definition of $H_i(u)$, the above function does not have poles at $\{\pm x_{j,t}\}$. Hence, by the same calculation as in the proof of Lemma \ref{lem:H_i}, we get
\begin{align*}
        &\bigg(-\frac{2\hbar^2u}{(u-x_{j,t})(u+x_{j,t})}H_i(u)\bigg)^\circ\\
        =&\sum_{r=1}^{v_i}\bigg(\frac{\hbar(-2\hbar-2x_{i,r})}{(x_{i,r}+\frac{\hbar}{2})^2-x_{j,t}^2}\big(\frac{1}{u+\frac{\hbar}{2}+x_{i,r}}+\frac{1}{u-\frac{\hbar}{2}-x_{i,r}}\big)\xi^+_{i,r}\xi^-_{i,r}\\
        &+\frac{\hbar(-2\hbar+2x_{i,r})}{(x_{i,r}-\frac{\hbar}{2})^2-x_{j,t}^2}\big(\frac{1}{u+\frac{\hbar}{2}-x_{i,r}}+\frac{1}{u-\frac{\hbar}{2}+x_{i,r}}\big)\xi^-_{i,r}\xi^+_{i,r}\bigg)+\frac{2\hbar^2}{ux_{j,t}^2}\xi_i^2.
\end{align*}
Hence, the coefficient of $\xi_{j,t}^+$ on both sides of \eqref{equ:serre} matches. By changing the variable $x_{j,t}$ to $-x_{j,t}$, the coefficient of $\xi_{j,t}^-$ will also match. 

Now let us consider the case of $t=0$. Then $\epsilon(j)=-$, and the coefficient of $\xi_j$ on the right hand side is 
\[\bigg(\hbar^2\frac{\xi_j}{-u+\frac{\hbar}{2}}-\hbar^2\frac{\xi_j}{u+\frac{\hbar}{2}}\bigg)H_i(u)=-\frac{2\hbar^2u}{u^2-\frac{\hbar^2}{4}}\xi_jH_i(u).\]
By letting $x_{j,t}=\frac{\hbar}{2}$ in the above calculation, we get that the coefficients of $\xi_j$ also match. This finishes the proof of the Serre relation. Hence, this concludes the proof of Theorem \ref{thm:gklo}.

\bibliographystyle{alpha}
\bibliography{Nonco.bib}

\end{document}